\documentclass[reqno]{amsart}

\usepackage{enumerate}
\usepackage{ifpdf}
\usepackage{amsmath}
\usepackage{amsfonts}
\usepackage{amssymb}
\usepackage{amsthm}
\usepackage[hyperfootnotes=false]{hyperref}
\usepackage{setspace}
\usepackage{amsrefs}
\usepackage{graphicx}
\usepackage{nicefrac}

\newcommand{\ignore}[1]{}

\newcommand{\abs}[1]{\left\lvert {#1} \right\rvert}
\newcommand{\norm}[1]{\left\lVert {#1} \right\rVert}

\newcommand{\C}{{\mathbb{C}}}
\newcommand{\R}{{\mathbb{R}}}
\newcommand{\Z}{{\mathbb{Z}}}
\newcommand{\N}{{\mathbb{N}}}

\newcommand{\CP}{{\mathbb{CP}}}

\newcommand{\sH}{{\mathcal{H}}}

\newcommand{\sI}{{\mathcal{I}}}

\newtheorem{thm}{Theorem}[section]

\newtheorem{prop}[thm]{Proposition}

\newtheorem{cor}[thm]{Corollary}

\newtheorem{lemma}[thm]{Lemma}

\theoremstyle{definition}
\newtheorem{defn}[thm]{Definition}

\theoremstyle{remark}

\author{Ji\v{r}\'{\i} Lebl}
\thanks{The first author was in part supported by NSF grant DMS 0900885.}
\address{Department of Mathematics, University of Illinois
at Urbana-Champaign,
Urbana, IL 61801, USA}
\email{jlebl@math.uiuc.edu}

\author{Han Peters}
\thanks{The second author was in part supported by NSF grant DMS 0757856.}
\address{Korteweg De Vries Institute for Mathematics, University of Amsterdam, Science Park 904,
1098 XH Amsterdam,
The Netherlands,}
\email{h.peters@uva.nl}

\date{December 21, 2010}

\ifpdf
\hypersetup{
  pdftitle={Polynomials constant on a hyperplane and CR maps of hyperquadrics},
  pdfauthor={Jiri Lebl and Han Peters}
}
\fi

\title[Polynomials constant on a hyperplane and maps of hyperquadrics]%
{Polynomials constant on a hyperplane and CR maps of hyperquadrics}

\begin{document}

%\doublespace

\begin{abstract}
We prove a sharp degree bound for polynomials constant on a hyperplane with a
fixed number of distinct monomials for dimensions 2 and 3.  We study the
connection with monomial CR maps of hyperquadrics and prove similar bounds in
this setup with emphasis on the case of spheres.  The results support
generalizing a conjecture on the degree bounds to the more general case of
hyperquadrics.
\end{abstract}


\maketitle

%\enlargethispage{\baselineskip}

%%%%%%%%%%%%%%%%%%%%%%%%%%%%%%%%%%%%%%%%%%%%%%%%%%%%%%%%%%%%%%%%%%%%%%%%%%%%%

\section{Introduction} \label{section:intro}

We are interested in the relationship between the complexity of real algebraic expressions on the one hand, and geometric properties on the other. Our main goal is to generalize a result of D'Angelo, Kos and Riehl. In \cite{DKR} they considered a class of real polynomials in two variables, and gave a precise quantitative description of the relationship between the degree of the polynomials, and the number of their nonzero coefficients. We are interested in generalizing their results to larger classes of polynomials, and in particular to higher dimensions. The classes of polynomials we will consider arise naturally in CR geometry, and our results give new degree bounds for holomorphic mappings between balls. We will discuss the relationship to CR geometry in more detail later.

The number of non-zero coefficients of real polynomials is known to bound
the geometry in different ways as well. For example, Descartes' rule of
signs implies that the number of real roots of a real polynomial is bounded
by a function depending only on the number of its nonzero coefficients. For
a discussion of these types of results, we refer the reader to the book \emph{Fewnomials} \cite{fewnomials} by Khovanski{\u\i}. Although the results in this article are not directly related those discussed in \cite{fewnomials}, important techniques that are used in this article can be found in the book as well.

Let us describe our results more precisely. Let $p$ be real polynomial of
degree $d$ in the variables
$x = (x_1,x_2,\ldots,x_n)$ such that $p(x) = 1$ when
$s(x) = \sum_{k=1}^n x_k = 1$.  Let $N(p)$ be the number of distinct
monomials in $p$.  In general there exists no function $C$ such that the inequality
\begin{equation} \label{theest}
d=\deg(p) \leq C(n,N(p)) .
\end{equation}
holds for all polynomials. For example, the polynomial
$p(x)=x_1^k(s(x)-1)+1$ is of degree $k+1$ for $k$ arbitrary,
but $N(p) = n+2$ and $p=1$ on $s=1$. 

A condition that arises naturally in CR geometry, where this problem
originated, is to require all coefficients of $p$ to be positive. In this
case the problem has been completely solved when $n=2$: for each $N$ the best
possible $C(2,N) = 2N-3$ was given by D'Angelo, Kos and Riehl \cite{DKR}. We note that when $n=1$ then $x^d = 1$ whenever $x=1$ for arbitrary $d$.
Thus when $n=1$, no function $C$ such that \eqref{theest}
holds can exist even when the coefficients are required to be positive.
D'Angelo has conjectured the best possible bound $C(n,N) = \frac{N-1}{n-1}$
for $n \geq 3$. In this paper we
prove his conjecture for $n=3$, and we give further reasons to suggest that
the conjecture also holds for $n \geq 4$.

We will give a condition of {\emph{indecomposability}}, which is weaker than
positivity of the coefficients, and show that for $n=2$ this weaker condition
implies the same estimate. In future work we hope to show that
indecomposability
gives the same estimates as positivity for all $n \geq 2$. To define
indecomposability we will work in a projective setting. This setting, which we
will discuss in the next section, is more pleasant to work in because of added
symmetry, and we will prove all our results in the projective setting.

When $n=3$ we have not yet been able to prove that the bound achieved for
polynomials with positive coefficients is also the best bound for polynomials
that satisfy indecomposability.  Under the additional mild assumption of
{\emph{no overhang}}, defined in \S~\ref{section:projprob},
we are able to achieve the same bound for $n=3$.
We note that positivity of the coefficients
implies the no-overhang condition.

We will say that an estimate \eqref{theest} is sharp
if no better estimate of the form \eqref{theest} is possible.  For
polynomials with positive coefficients, the following theorem is known.

\begin{thm} \label{thm:dkr-dlp}
Let $p \in \R[x_1,\ldots,x_n]$ be a polynomial of degree $d$ with positive
coefficients such that $p(x) =1$ whenever $s(x) = 1$.  Let $N=N(p)$.
\begin{enumerate}[(i)]
\item (D'Angelo-Kos-Riehl~\cite{DKR}) If $n=2$
we have the sharp estimate
\begin{equation}
d \leq 2N-3 .
\end{equation}
\item (D'Angelo-Lebl-Peters~\cite{DLP}) If $n \geq 2$
then
\begin{equation} \label{highnbounddlp}
d \leq  \frac{4}{3} \, \frac{2N-3}{2n-3} .
\end{equation}
\end{enumerate}
\end{thm}

The second estimate is not sharp.
Furthermore, in \cite{DLP}
we also showed that for $n$ sufficiently larger than $d$ (or alternatively
when $N-n$ is sufficiently small) we have the sharp bound
\begin{equation} \label{wantedbound}
d \leq \frac{N-1}{n-1} .
\end{equation}
In fact, we have classified all $p$ that give equality in
\eqref{wantedbound} for $n$ and $d$ as above.

We extend the results above in two directions.  First, we prove
similar estimates for polynomials under the weaker assumption of
indecomposability.  Second, we prove a sharp estimate for polynomials with
positive coefficients when $n=3$.

\begin{thm} \label{thm:nonhomog}
Let $p \in \R[x_1,\ldots,x_n]$ be a polynomial of degree $d$ such that
$p(x) =1$ whenever $s(x) = 1$.  Let $N=N(p)$.
\begin{enumerate}[(i)]
\item If $n=2$
and $p$ is indecomposable then we have the sharp estimate
\begin{equation}
d \leq 2N-3 .
\end{equation}
\item If $n=3$
and all coefficients of $p$ are positive then we have the sharp estimate
\begin{equation}
d \leq \frac{N-1}{2} .
\end{equation}
\item If $n=3$,
$p$ is indecomposable, and $p$ satisfies the no-overhang condition
then we have the sharp estimate
\begin{equation}
d \leq \frac{N-1}{2} .
\end{equation}
\item If $n \geq 2$ and
$p$ is indecomposable
then
\begin{equation} \label{highnboundirred}
d \leq  \frac{4}{3} \, \frac{2N-3}{2n-3} .
\end{equation}
\end{enumerate}
\end{thm}

With the results of this paper we conjecture that \eqref{wantedbound}
holds not only
under the positivity condition, but also under the weaker condition of
indecomposability for $n \geq 3$.  It is not hard to construct examples that
achieve equality in \eqref{wantedbound} and hence the bound is sharp if true.

We will study the connection with CR geometry in \S~\ref{section:CR}.  Our
main results can be stated as follows.  Let $Q(a,b) \subset \CP^{a+b}$ be the
hyperquadric with $a$ positive and $b$ negative eigenvalues
as a subset of complex projective space.  A monomial map is a map whose
components are single monomials (in homogeneous coordinates).
We say a CR map $f \colon Q(a,b) \to
Q(c,d)$ is \emph{monomial-indecomposable} if it is not of the form
$f = f_1 \oplus f_2$, in homogeneous coordinates, for two monomial CR maps
$f_1 \colon Q(a,b) \to Q(c_1,d_1)$ and
$f_2 \colon Q(a,b) \to Q(c_2,d_2)$.  Note that $Q(n,0) = S^{2n-1}$ is the
sphere.
We say that a map $f$ has \emph{linearly independent components} if
the components of the map in homogeneous coordinates
are linearly independent.

\begin{thm} \label{thm:CRbounds}
$~$
\begin{enumerate}[(i)]
\item \label{thm:CRbounds:i}
Let $f \colon Q(a,b) \subset \CP^2 \to Q(c,d) \subset \CP^N$
be a monomial CR map with linearly independent components that is
monomial-indecomposable.  Then we have the sharp estimate
\begin{equation}
\deg(f) \leq 2N-3 .
\end{equation}
\item \label{thm:CRbounds:ii}
Let $f \colon S^{5} \subset \C^3 \to S^{2N-1} \subset \C^N$
be a monomial CR map.  Then we have the sharp estimate
\begin{equation} \label{n2monbound}
\deg(f) \leq \frac{N-1}{2} .
\end{equation}
\item \label{thm:CRbounds:iii}
Let $f \colon Q(a,b) \subset \CP^n \to Q(c,d) \subset \CP^N$, $n \geq 2$, be
a monomial CR
map with linearly independent components
that is monomial-indecomposable.
There exists a number $C(n,N)$
depending only on $n$ and $N$ such that
\begin{equation}
\deg(f) \leq  C(n,N) .
\end{equation}
\end{enumerate}
\end{thm}

Note that both conditions of monomial-indecomposable and linear independence of
the components are needed.
The value we obtain for $C(n,N)$ in item~\eqref{thm:CRbounds:iii} of
Theorem~\ref{thm:CRbounds} is far from optimal.  See
\S~\ref{section:higherdim} for the precise bound we proved.

We conjecture that the sharp bounds for $n=2$ and $n=3$ hold not only for
monomial, but also for all rational CR maps of hyperquadrics
that are indecomposable (into rational CR maps) for all $n$.
For $n \geq 3$ we conjecture
the sharp bound \eqref{wantedbound} holds for the same class of
functions.

{\emph{Acknowledgement: The authors would like to thank John D'Angelo for introducing this problem to us in 2005, and for his help and guidance since then. Part of the progress was booked during workshops at MSRI and AIM, the authors would like to thank both institutes.}}

%%%%%%%%%%%%%%%%%%%%%%%%%%%%%%%%%%%%%%%%%%%%%%%%%%%%%%%%%%%%%%%%%%%%%%%%%%%%%

\section{Projectivized problem} \label{section:projprob}

We will first formulate a projectivized version of the problem.  Take
a polynomial $p(x)$ such that $p(x) = 1$ on $s(x) = 1$.  Write $p(x) =
\sum_{k=0}^d p_k(x)$ for homogeneous polynomials $p_k$ of degree $k$.
Now homogenize
$p(x)-1$ and $s(x)-1$ to obtain that
\begin{equation}
\left(\sum_{k=0}^d p_k(x) \, t^{d-k} \right) - t^d
= 0
\qquad \text{whenever} \qquad
x_1 + x_2 + \cdots + x_n - t = 0 .
\end{equation}
We will generally replace $t$ with $-t$.  Then, changing notation, we
look at homogeneous polynomials
$P(X)$, $X \in \R^{n+1}$, such that
\begin{equation}
P(X) = 0 \qquad \text{on} \qquad S(X) := X_0 + X_1 + \cdots + X_n = 0 .
\end{equation}
This substitution
makes the problem more symmetric and easier to handle.  On the other hand,
in this formulation the condition that $p(x)$ has had only positive
coefficients is harder to see.  We therefore need a substitute.

\begin{defn} \label{defn:basicdefs}
Let $\sI = \sI(n)$ denote the set of homogeneous polynomials
$P \in \R[X_0,X_1,\ldots,X_n]$ such that $P(X) = 0$ on $S(X) = 0$.

We say $P(X)$ has \emph{p-degree} $d$ if $d$ is the smallest
integer such that there exists a monomial $X^\alpha$  (where $\alpha$ is a
multi-index)
and a polynomial $R(X)$ of degree $d$ such that $P(X) = X^\alpha R(X)$.
Note that if the monomials of $P$ have no common divisor then
we can take $\alpha = (0,\ldots,0)$ and the p-degree of $P$
is equal to the degree of $P$.
Let $\sI(n,d)$ denote the polynomials in $\sI(n)$ of p-degree $d$.

We will say that $P \in \sI$ is \emph{indecomposable} if we cannot write $P(X) =
P_1(X) + P_2(X)$, with $P_1,P_2 \in \sI$, such that $P_1$ and $P_2$ are two
nontrivial polynomials with no monomials in common.

We define $N(P)$ be the number of distinct monomials in $P$.
\end{defn}

\begin{defn}
Let $\sH(n,d)$ denote the set of polynomials $p \in \R[x_1,\ldots,x_n]$ of
degree $d$ that have only non-negative coefficients, and such that $p(x) = 1$
on $s(x) = x_1+\cdots+x_n = 1$.

Just as in the previous definition, we define $N(p)$ be the number of
distinct monomials in $p$.
\end{defn}

Suppose we start with $p \in \sH(n,d)$.  We homogenize with $t$ and then let
$t= - X_0$.  The induced polynomial $P(X) \in \sI$ is indecomposable.  Were it
decomposable, the monomial $X_0^d$ could belong only to one of the summands.
The other summand would then correspond (after changing $X_0$ back to $-t$)
to a polynomial with nonnegative coefficients that is zero on $s(x) = 1$.  It
is not hard to see that such a polynomial must be identically zero.

Let us state some obvious results.

\begin{prop}
{\ }
\begin{enumerate}[(i)]
\item The only indecomposable $P \in \sI(0)$ are
polynomials $c X_0^d$ for $d \geq 1$, in particular $N(P) = 1$ and
$\operatorname{p-degree}(P) = 0$.
\item There exist indecomposable $P \in \sI(1)$
with $N(P) = 2$ and arbitrarily large p-degree.
\end{enumerate}
\end{prop}

To construct examples for the second item, note that
$P(X) = X_0^d+(-1)^{d+1} X_1^d
\in \sI(1)$ is indecomposable, $N(P) = 2$, and p-degree is $d$.
Matters become more complicated for $n=2$.
We will prove the following theorem.

\begin{thm} \label{maindim2thm}
If $P \in \sI(2)$ is indecomposable and of p-degree $d$, then
\begin{equation}
d \leq 2 N(P)-5 .
\end{equation}
\end{thm}

Instead of directly proving this result, we will prove a stronger lemma,
which we will also need for $n=3$, and which will imply the result.  To state
this lemma, we need to first give the discrete geometric setup.

We define the homogeneous polynomial $Q$ by:
\begin{equation}
Q(X) = \frac{P(X)}{S(X)} .
\end{equation}
We will generally work with $Q$ rather than $P$.  In fact, we will even
forget the sizes of the coefficients of $Q$ and deal only with the signs
in a so-called Newton diagram.

The condition of indecomposability will turn out to be a connectedness
condition on the Newton diagram, when the problem is translated into a
problem in discrete geometry below.  The p-degree will then be the diameter of
the set of nonzero signs in the Newton diagram.

\begin{defn}
For $m \in \Z^n$, we will write $\abs{m} = m_1 + m_2 + \cdots + m_n$.

Let $P = S \cdot Q$ as above, where $P$ is of degree $d$.
Let $\gamma \colon \Z^n \to \Z^{n+1}$
be the map
\begin{equation}
\gamma(m) := \left(d-1-\abs{m}, ~ m \right) .
\end{equation}
We define the function $D \colon \Z^n \to \{ 0, P, N \}$ as follows.
Let $C_{\gamma(m)}$ be the coefficient of $X^{\gamma(m)}$ in $Q(X)$
when we treat $\gamma(m)$ as a multi-index.  Then
\begin{equation}
D(m) :=
\begin{cases}
P & \text{ if $C_{\gamma(m)} > 0$, } \\
0 & \text{ if $C_{\gamma(m)} = 0$, } \\
N & \text{ if $C_{\gamma(m)} < 0$. }
\end{cases}
\end{equation}
We call $D$ the \emph{Newton diagram} of $Q$, and we will say that
$D$ is the Newton diagram corresponding to $P$.

We will call the $m \in \Z^n$ points of $D$, and we will call
$m$ a $0$-point if $D(m)=0$, a $P$-point if $D(m) = P$ and
an $N$-point if $D(m) = N$.  We say that the monomial $X^{\gamma(m)}$
is the monomial associated to $m \in \Z^n$ and vice-versa.
We will often identify points of $\Z^n$ with the associated monomials.
\end{defn}

Therefore, the Newton diagram of $Q$ is an array of $P$s, $N$s, and $0$s,
one for each coefficient of $Q$.
For example when $n=2$,
the monomial $C_{abc} X_0^a X_1^b X_2^c$
corresponds to the point with coordinates $(b,c)$.  In the Newton diagram we
let $D(b,c) = 0$ if $C_{abc} = 0$, we let $D(b,c) = P$ if $C_{abc} >
0$, and we let $D(b,c) = N$ if $C_{abc} < 0$.  Note that we include negative
powers in the Newton diagram, even though $D(m) = 0$ any time $m$ is not in
the positive quadrant.

We will generally ignore points $m \in \Z^n$ where $D(m) = 0$.  We can give
graphical representation of $D$ by drawing a lattice, and then
drawing the values of $D$ in the lattice.  For convenience,
when drawing the $n=2$
lattice, we put $(0,0)$ at the origin, and then let $(0,1)$ be directed
at angle $\frac{\pi}{3}$ and $(1,0)$ at angle $\frac{2\pi}{3}$.  Similarly we depict
the diagram for $n=3$.
In the figures we will use light color circles and spheres for for positive
coefficients and dark color circles and spheres for negative coefficients.
We will not draw the circles and spheres corresponding to the zero
coefficients.
See Figure~\ref{fig:newton} for sample diagrams.

\begin{figure}[h!t]
\begin{center}
\begin{minipage}[b]{0.49\linewidth}
\centering
\includegraphics{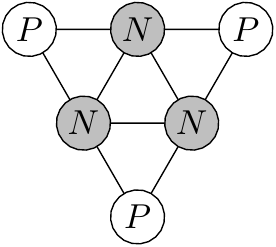}
\end{minipage}
\begin{minipage}[b]{0.49\linewidth}
\centering
\includegraphics{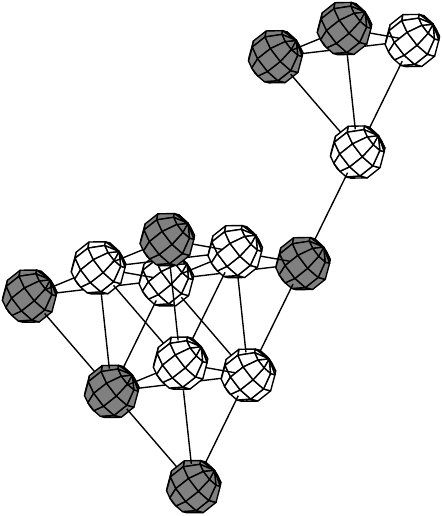}
\end{minipage}
\caption{Newton diagram for $n=2$ for
$Q(X)=X_0^2+X_1^2+X_2^2
- X_0 X_1 - X_1 X_2 - X_0 X_2$ (on left)
and a sample diagram for $n=3$ (on right).\label{fig:newton}}
\end{center}
\end{figure}

\begin{defn}
We say two points in the Newton diagram are \emph{adjacent} if the
corresponding monomials $X^\alpha$ and $X^\beta$ are such that
there exist $k$ and $j$ such that $X_j X^\alpha = X_k X^\beta$.
Between any two points (or any two monomials)
we can define the distance
$\operatorname{dist}(X^\alpha,X^\beta)$ as 1 if the corresponding
monomials are adjacent, and if they are not adjacent as the length
of the
shortest path along adjacent monomials.

The {\emph{support}} $K \subset \Z^n$ of a Newton diagram $D$ is the set
of nonzero points.  That is, $K = D^{-1}(\{ P, N \})$.  As $K$ cannot
contain any negative coordinates, we will generally consider
$K$ a subset of $\N_0^n$ (where $\N_0 = \{ 0 \} \cup \N$).

We say that the support $K$ of a Newton diagram $D$ is \emph{connected} if
any two points of $K$ are joined by a path of
adjacent points of $K$.

For a set $K \subset \Z^n$, we let $\widehat{K} \subset \Z^n$ be the smallest simplex that contains $K$. More precisely, $\widehat{K} \subset \Z^n$ is the smallest
set such that $K \subset \widehat{K}$ and for some
$a = (a_1,\ldots,a_n) \in \Z^n$
and some $k \in \N$ we have
\begin{equation}
\widehat{K} = \{x \in \Z^n \mid x_j \geq a_j \;, \; \abs{x} \leq k\}.
\end{equation}
When $K \subset \N_0^n$ then the diameter of $\widehat{K}$ in the distance defined above
is easily seen to be $k-\abs{a}$.  We define the \emph{size} of $K$ as the diameter of
$\widehat{K}$ plus one.
\end{defn}

The adjacency between points in the figures is marked by drawing a line
connecting the two circles or spheres. If $P \in \sI(n)$, then
we will show that p-degree of $P$ is equal to the size of $K$ where $K$ is
the support of the Newton diagram corresponding to $P$.

\begin{defn}
Let $D \colon \Z^n \to \{ 0, P, N \}$ be a Newton diagram.
Let $X^\alpha$ be a monomial of degree $d$.  Let $E = E(\alpha)$
be the set of points in $\Z^n$ defined by
\begin{equation}
E := \{ m \in \Z^n \mid X_k X^{\gamma(m)} = X^\alpha \text{ for some $k$} \} .
\end{equation}
We call $E$
a \emph{node} if the direct image
$D(E) = \{ P \}$, $D(E) = \{ N \}$,
$D(E) = \{ 0, P \}$, or $D(E) = \{ 0, N \}$.
Do note that we allow negative indices in $E$.
For a diagram $D$ let
\begin{equation}
\#(D) := \text{(number of nodes in $D$)} .
\end{equation}
\end{defn}

See Figure~\ref{fig:newtonnodes} for examples of diagrams with nodes marked.
In dimension 2, the nodes are marked with a triangle where each vertex of the
triangle touches a point of the node.  In dimension 3 we mark the nodes
similarly with a simplex.

\begin{figure}[h!t]
\begin{center}
\begin{minipage}[b]{0.39\linewidth}
\centering
\includegraphics{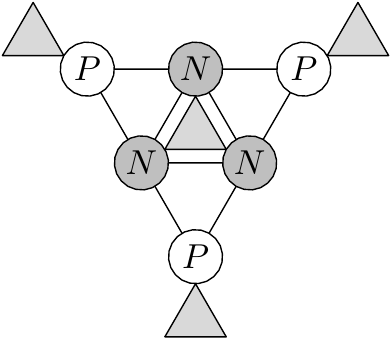}
\end{minipage}
\begin{minipage}[b]{0.59\linewidth}
\centering
\includegraphics{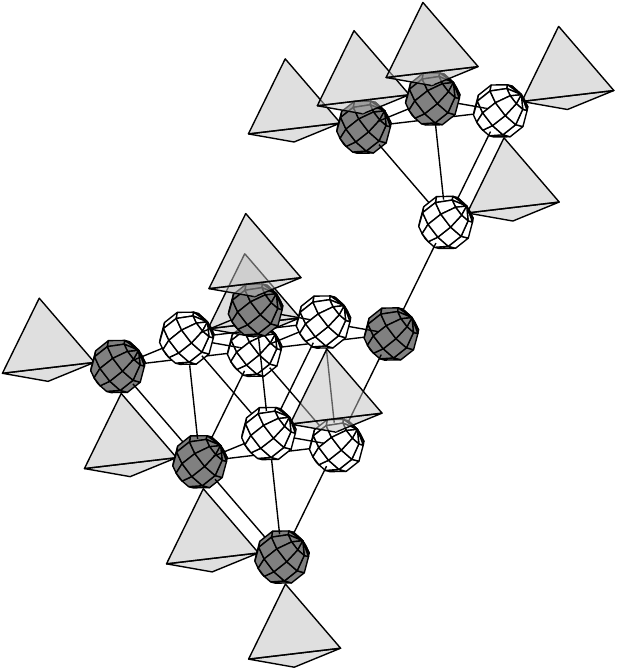}
\end{minipage}
\caption{Nodes in Newton diagrams for $n=2$ (on left) and
$n=3$ (on right).\label{fig:newtonnodes}}
\end{center}
\end{figure}

The number of nodes $\#(D)$ functions as a lower bound for the number of
non-zero coefficients for a polynomial in $\sH(n,d)$, or rather, as a
generalization thereof. Let $p \in \sH(n,d)$ and let $P$ be the induced
homogeneous polynomial with corresponding Newton diagram $D$. If the diagram
$D$ has a node at the points of the form $\frac{X_0}{X_j} X^\alpha$, with
$\alpha = (\alpha_0, \alpha^\prime)$,  then it is easy to see that for the
polynomial $p$, the coefficient of the monomial $x^{\alpha^\prime}$ cannot be
zero. The non-negativity of the coefficients of $p$ puts a restriction on the
kind of nodes that can occur in each mutually connected collection of nonzero
points,
but as we are also interested in allowing coefficients of the polynomials $p$
to be negative we will ignore this restriction as much as possible.

In two dimensions it will be sufficient to obtain only the very mild
restrictions on the support that are given by the next two lemmas.

\begin{lemma}\label{lemma:connectedness1}
Let $P \in \sI(n)$.  Write
$P = S \cdot Q$.
Let $D$ be the Newton diagram of $Q$, with support $K$.
\begin{enumerate}[(i)]
\item
If $P$ is indecomposable, then $K$ is connected.
\item
The size of $K$ is equal to the p-degree of $P$.
\end{enumerate}
\end{lemma}

\begin{proof}
Let us prove the first item.
Suppose that $K$ is disconnected.  Then take $K_1,\ldots,K_r$
be the connected components.  Define the $Q_j$ by taking those
monomials of $Q$ that are in $K_j$.  As if $j\not=k$,
then no points in $K_k$ is adjacent to any point of $K_j$.
Therefore $Q_k S$ has no monomials in common with $Q_j S$.
Thus define $P_j = Q_j S$.  We see that $P = P_1 + \cdots + P_r$
and $P$ is decomposable.

Now let us prove the second item.
By taking $a = (0,\ldots,0)$ in the definition of $\widehat{K}$ and
the size of $K$,
it is not hard to see that the size of $K$ is at least the p-degree of
$P$.  Suppose that $\widehat{K}$ does not contain any points $m \in \Z^n$
with $m_1 \leq k$ for some $k$.  Then $Q$ is divisible by $X_1^{k+1}$.  The
result follows by iterating through all the variables.
\end{proof}

\begin{lemma}\label{lemma:connectedness}
Let $p \in \sH(n,d)$, let $P = S \cdot Q$ be the induced homogeneous
polynomial and $D$ the Newton diagram of $Q$, with support $K$. Then $K$ is
connected, contains $(0, \ldots, 0)$ and is of size $d$.
\end{lemma}

\begin{proof}
First we note that $P$ is indecomposable.  Suppose that $P$ were
decomposable.
That would in fact imply that there would exist two nontrivial
polynomials $p_1$
and $p_2$ with positive coefficients such that $p_1 + p_2 = p$,
$p_1$ and $p_2$ are constant on the hyperplane $s=1$, and have no monomials
in common.  In fact, the constant term in $p-1$ can only be in one of
$p_1$ or $p_2$ and hence one of them is zero on the hyperplane.  So suppose
that $p_2 = 0$ on $s=1$.  As all coefficients of $p_2$ are positive,
then they are all zero.  Hence $p_2 \equiv 0$.  Therefore $P$ is
indecomposable.

Thus by Lemma~\ref{lemma:connectedness1}, $K$ is connected.  If we let
$q = \frac{p-1}{s-1}$.  If $p(0) = 1$, then we obtain a contradiction
as above, $p-1$ would have all positive coefficients and be
zero on $s=1$.  Therefore $p(0) \not= 1$.  Then it is not hard to see
that $q$ must contain a nonzero constant.  Therefore so does $Q$ and
hence $(0,\ldots,0) \in K$.

Once $(0,\ldots,0) \in K$, we see that the size of $K$ cannot be any smaller
than $d$.  Since for all points $m \in K$ we have $\abs{m} < d$, then
$K$ is of size $d$.
\end{proof}

In fact,
in two dimensions we will be able to prove the estimate on $\#(D)$
without requiring that $(0,0) \in K$.  Compare
Theorem~\ref{maindim2newtonthm} and Proposition~\ref{twodim}.

Our proof in three dimensions requires stronger restrictions on $K$. We make
the following definition.

\begin{defn}
A set $K \subset \Z^2$ has \emph{left overhang} if there exists
a point $(a,b) \in K$, $(a,b) \not= (0,0)$ such that
$(a,b-1) \notin K$ and $(a-1,y) \notin K$ for all $y \geq b$.

A point $(a,b) \in K$ is \emph{right overhang} if it is a left overhang
after swapping variables.  We say simply that $(a,b)$ is an \emph{overhang}
if it is a left or a right overhang.

Let $K \subset \Z^3$.
Define $\pi_j \colon \Z^3 \to \Z^2$, $j=1,2,3$, by
$\pi_1(m) = (m_1,m_2+m_3)$,
$\pi_2(m) = (m_1+m_3,m_2)$, and
$\pi_3(m) = (m_1+m_2,m_3)$.  Then
we say that
$m \in K$ is an overhang if $\pi_j(m)$ is an overhang of
$\pi_j(K)$ for some $j = 1,2,3$.

When no overhang exists in
$K \subset \Z^n$ we say that $K$ satisfies the \emph{no-overhang}
condition.  We will sometimes say that a Newton diagram $D$ satisfies
the no-overhang condition,
which means that the support $K$ of $D$ has no overhang.  Similarly,
when we say that $P \in \sI(n)$ satisfies the no-overhang condition
we mean that the support of the corresponding Newton diagram
satisfies the no-overhang condition.
\end{defn}

The projection $\pi_1$ is equivalent to setting the 2nd and 3rd variable
equal to each other in the underlying polynomial.  Similarly for $\pi_2$ and
$\pi_3$.
The Figure~\ref{fig:noover} illustrates a diagram with an overhang.  The
overhang is marked with a circle in the figure.  The second view is along one
axis signifying the projection of the support by a $\pi_j$ for some $j$.
The overhang in 2 dimensions is then easy to see.

\begin{figure}[h!t]
\begin{center}
\includegraphics{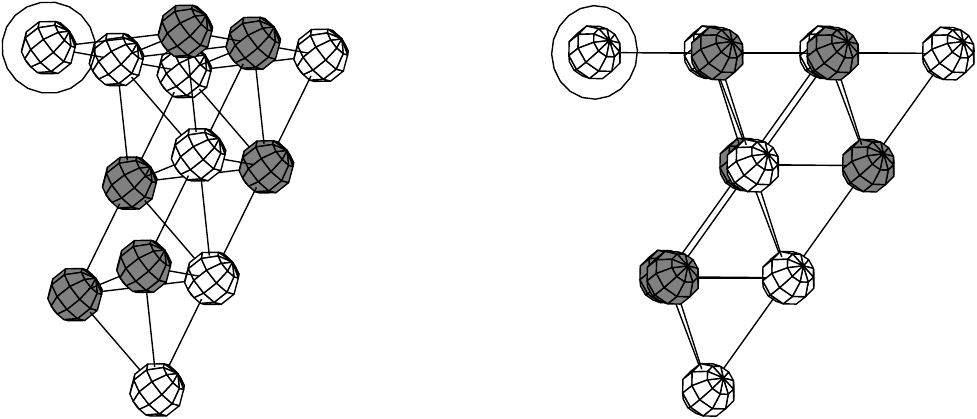}
\caption{Two views of a diagram violating the \emph{no-overhang}
condition.\label{fig:noover}}
\end{center}
\end{figure}

\begin{lemma}\label{lemma:overhang}
Let $p \in \sH(n,d)$, $2 \leq n \leq 3$, and $P, Q, D, K$ be as in
Lemma~\ref{lemma:connectedness}.  Then $K$ has no overhang.
\end{lemma}

\begin{proof}
We start with $n=2$.  Let $p \in \sH(2,d)$.  Suppose that $(a,b) \in K$
be an overhang.  Without loss of generality assume $(a,b)$ is
a left overhang.  Since $D(a-1,b) = D(a,b-1) = 0$, then we have a node
and $p$ contains the monomial $x_1^a x_2^b$.  As $(a,b) \not= (0,0)$,
the coefficient of $x_1^a x_2^b$ has to be positive.  Let $q =
\frac{p-1}{s-1}$, the coefficient of $x_1^a x_2^b$ must be negative.  Because
we are multiplying $q$ by $s-1 = x_1 + x_2 - 1$, we note that the coefficient
of $x_1^a x_2^{b+1}$ in $q$ must be negative as well, otherwise we would get
a negative coefficient of $x_1^a x_2^{b+1}$ in $p$.  But that means
that the coefficient of
$x_1^a x_2^{k}$ in $q$ for all $k \geq b$ must be negative.  But $q$ is a
polynomial and hence for some $k$, the coefficient is zero.  Then we get
a negative term in $p$ which is a contradiction.
Therefore the lemma holds for $n=2$.

Now let $n=3$.  Let $p \in \sH(3,d)$.
Suppose for contradiction that $(a,b,c) \in K$ is an overhang.
Without loss of generality suppose that $\pi_1(a,b,c) = (a,b+c)$ is
an overhang of $\pi_1(K)$.
For some $t \in (0,1)$, let $\tilde{p}(x_1,x_2) = p(x_1,tx_2,(1-t)x_2)$.
Define $\tilde{q} = \frac{\tilde{p}-1}{s-1}$.
We can pick $t$ such that
the coefficient of $x_1^a x_2^{b+c}$ in $\tilde{q}$ is nonzero,
which we can certainly do.
Let if we homogenize $\tilde{p}$ to
create $\tilde{P}$, find $\tilde{Q}$ and the Newton diagram $\tilde{D}$,
then if $\tilde{K}$ is the support of $\tilde{D}$, then
it is easy to see that $\tilde{K} \subset \pi_1(K)$.
Furthermore, $(a,b+c) \in \tilde{K}$.
If $(a,b+c)$ was an overhang of $\pi_1(K)$,
then it is an overhang of $\tilde{K}$.  The fact that
$\tilde{p} \in \sH(2,d)$ leads to a contradiction.
\end{proof}

In the next two sections we will often refer to directions of a support $K$,
imagining the $(0, \ldots, 0)$ point at the {\emph{bottom}} and the points
$m$ with $\abs{m} = d-1$ at the \emph{top}.  Note that points
at the top correspond to monomials of degree $d-1$ in nonhomogeneous
coordinates obtained by setting $X_0 = 1$.
Similarly by
{\emph{up}} (resp.\@ {\emph{down}})
we will mean the direction of increasing
(resp.\@ decreasing) $\abs{m}$ (that is, degree in nonhomogeneous
coordinates).  By a {\emph{side}} of $K$ we mean a collection of
points
for which a fixed coordinate is equal to zero. In the two dimensional case we
will use {\emph{left}} and {\emph{right}} without worrying about which is
which.  By the $k$th \emph{level} we will mean points $m$ where $\abs{m} =
m_1 + \cdots + m_n = k$.

%%%%%%%%%%%%%%%%%%%%%%%%%%%%%%%%%%%%%%%%%%%%%%%%%%%%%%%%%%%%%%%%%%%%%%%%%%%%%

\section{Two-dimensional results}\label{twodimensions}

Let us start by recalling some ideas of the proof of the $2$-dimensional
result by D'Angelo, Kos, and Riehl. We start with a polynomial $p(x,y)$ with
non-negative coefficients that satisfies $p = 1$ when $s = x+y = 1$. We
assume that $p$ is of degree $d$ and we wish to estimate from below the
number of non-negative coefficients of $p$.

D'Angelo, Kos and Riehl then write $p = (s-1)q + 1$ for some polynomial $q$,
but instead we will use the induced homogeneous polynomial $P$ and the
notation $P = S \cdot Q$ as was introduced in the previous section. We obtain
the Newton diagram $D$  of $Q$ and as we discussed in the previous section,
the number of nodes of $D$ is a lower estimate for the number of positive
coefficients of $p$. The positivity of $p$ puts strong restrictions on the
diagram $D$, but for now we only require that the support $K$ of $D$ is
connected.  We will at first require that
$(0,0) \in K$ but we will remove this restriction later.

By manipulating the Newton diagram $D$ without increasing the number of
nodes, D'Angelo, Kos and Riehl were able to show that the number of nodes of
$D$ must be at least $\frac{d+5}{2}$. Since the node involving the points
$(0, 0)$, $(0,-1)$, and $(-1,0)$ does not correspond to a
positive coefficient of $p$ but all other nodes do, we see that $p$ must have
at least $\frac{d+3}{2}$ positive coefficients. Considering that generally
many of the coefficients of the polynomial $p$ are ignored when counting
nodes, it is truly remarkable that this argument gives the sharp bound for
$N(p)$.

Their method is the following: they make small
local changes to $D$ such that after every step the number of nodes has not
increased and the lowest and left most node has moved to the right or up.
After a finite number of steps all nodes must be at the top level and it is
then easy to give a lower estimate on the number of nodes.

Instead of following their argument we give a slightly different proof. The
reason is that the above argument cannot possibly work in dimensions $3$ and
higher. It is not hard to see that if all nodes are at the top level then the
estimate on the number of nodes would be at least quadratic in $d$. But since
it is easy to give examples where the number of coefficients of $p$ (and thus
also the number of nodes of $D$) is at most linear in $d$, it cannot always be
possible to push the nodes to the top level without increasing their number.
Fortunately the argument that we will give here is also a bit simpler than
the original proof.

Instead of manipulating $D$ to remove the lowest nodes, we will repeatedly
eliminate the lowest zero, and end up with a Newton diagram that consists
entirely of $P$s and $N$s.

\begin{prop}\label{twodim}
Let $D$ be a Newton diagram with support $K$ in two variables.
Suppose that $K$ is of size $d$,
is
connected, and contains $(0,0)$. Then
\begin{equation}
\#(D) \geq \frac{d+5}{2}.
\end{equation}
\end{prop}

The proof of Proposition \ref{twodim} will consist of two short lemmas.
We will say that a diagram $D^\prime$ is an \emph{extension} of $D$
if the support $K^\prime$ of $D^\prime$ contains the support $K$ of $D$
and such that $D^\prime|_K = D|_K$.

\begin{lemma}\label{filling}
Let $D$ and $K$ be as in Proposition \ref{twodim}.  Let $k < d$ and suppose that
$K$ contains all points $a$ with $\abs{a} \leq k-1$.  Then there exists an
extension $D^\prime$ of $D$ that has equal or fewer nodes, and whose support
$K^\prime$ contains all points $a$ with $\abs{a} \leq k$.
\end{lemma}

\begin{proof}
We change the $0$-points at level $k$ into alternating $N$- and $P$-points, one
connected group of $0$-points of level $k$ at a time.
Since there must be at least one
$P$- or $N$-point at the $k$-level, there are two cases that we should
consider: one is that the group has $N$- or $P$-points on both sides, and the
other is that the group extends to one of the boundaries.  In
the former case no new nodes can be created by changing the points and we are
done. In the latter case it is possible to create a node at the boundary, but
by choosing the boundary point correctly one boundary node is eliminated as
well and we are done.
See Figure~\ref{fig:filling}.  The circles with the dark borders
are the nonzero points that were added in a single step.
\end{proof}

\begin{figure}[h!t]
\begin{center}
\includegraphics{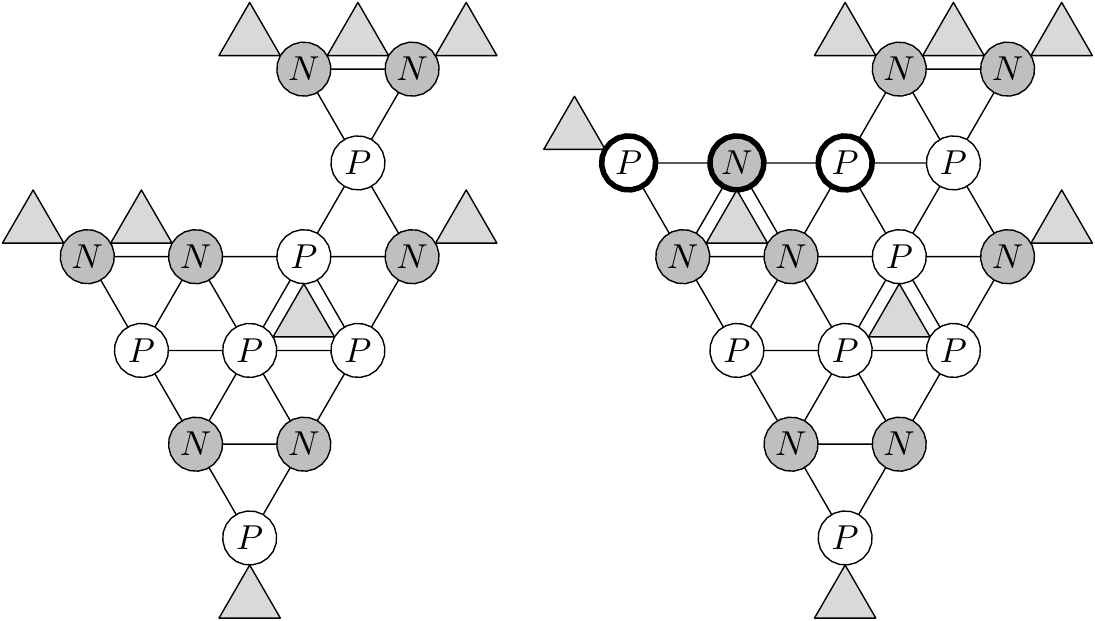}
\caption{Filling a 2-dimensional diagram.\label{fig:filling}}
\end{center}
\end{figure}

Using Lemma~\ref{filling} repeatedly we obtain a diagram with
support
$K = \{(a,b) \in \N_0^2 \mid a+b < d\}$.
Proposition~\ref{twodim} now follows from the following Lemma:

\begin{lemma}\label{dkr-lemma}
Let $D$ be a Newton diagram in two variables with support $K$ of size $d$
and suppose that $K$ is
maximal: it contains all $(a,b) \in \N_0^2$
with $a+b < d$. Then $D$ has at least
$\frac{d+5}{2}$ nodes.
\end{lemma}

\begin{proof}
Let us consider a horizontal row $\{(a,b) \in \N_0^2 \mid a+b = k\}$ at level $k < d-1$. Suppose that, starting with $(k,0)$ and ending with $(0,k)$, the sign changes $m$ times from $P$ to $N$ or vice-versa. Also suppose that there are $n$ sign changes on the next higher row $\{(a,b) \mid a+b = k+1\}$. Then there are at least $\frac{\abs{m-n}}{2}$ nodes involving points only on these two rows.

Now suppose that there are $l$ sign changes on the highest row $\{(a,b) \mid a + b = d-1\}$. Then there are exactly $d-l+1$ nodes on this highest row. Since there are zero sign changes on the lowest row $\{(0,0)\}$, it follows from the observation in the previous paragraph that there must be at least $\frac{l}{2}$ nodes involving some point $(a,b)$ with $1 < a+b < d-1$, as well as the one node involving the points $(0,0), (-1,0)$ and $(0,-1)$. Adding the three numbers gives a total of at least $d - \frac{l}{2} + 2$ nodes. The minimum number of nodes can therefore only be achieved when $l = d-1$, the maximum number of sign changes on the row $\{(a,b) \mid a+b = d-1\}$. Inserting $l=d-1$ shows that there are at least $\frac{d+5}{2}$ nodes.
\end{proof}

We end this section with the proof of Theorem~\ref{maindim2thm}.  We will
prove the following stronger result about Newton diagrams that at once implies
Theorem~\ref{maindim2thm}.

\begin{thm} \label{maindim2newtonthm}
Let $D$ be a Newton diagram with connected support $K$ of size $d$
in two variables. Then
\begin{equation}
\#(D) \geq \frac{d+5}{2}.
\end{equation}
\end{thm}

\begin{proof}
Without loss of generality we may assume
that the support $K$ contains points on both sides, that is
point $(a,0)$ and a point $(0,b)$.  By connectedness $K$ must
contain a connected path from the left side to the right. By the proof of
Lemma~\ref{filling} we can change all the $0$-points of $D$ above this
connected path into $P$- and $N$-points without increasing the number of
nodes. By switching the roles of $X_0$ and $X_1$ we then obtain a new Newton
diagram with the same number of nodes whose support is connected, of size $d$
and contains $(0,0)$, which by Proposition~\ref{twodim} completes the proof.
\end{proof}

%%%%%%%%%%%%%%%%%%%%%%%%%%%%%%%%%%%%%%%%%%%%%%%%%%%%%%%%%%%%%%%%%%%%%%%%%%%%%

\section{Preparing for three dimensions}

The main idea that we use for estimating the number of nodes of a
$3$-dimensional Newton diagram $D$ is to ignore the nodes that might occur in
the interior of the diagram and to only count the number of nodes that occur
on each of the faces of the support $K$. For each of these faces we can then use
two dimensional estimates. Of course, by counting the number of nodes on each
face separately and then adding the estimates we are in danger of counting
some nodes more than once, namely nodes involving points that lie on more than
one face. We therefore adjust the node count.

\begin{defn}
Let $D$ be a $2$-dimensional Newton diagram. A node that involves no zeroes
is called an {\emph{interior node}}, a node with exactly one zero is called
an {\emph{edge node}}, and a node with two zeroes is called a {\emph{vertex
node}}. When both zeroes of a vertex node lie at the bottom of the node, the
node is called a {\emph{bottom node}}.

The {\emph{weighted surface count}} of $D$ is the number of interior nodes, plus half
the number of edge and vertex nodes, but ignoring the number of corner nodes:
\begin{multline}
SC(D) := \text{(\# of interior nodes)}
+ \frac{1}{2} \times \Bigl(\text{(\# of edge nodes)} + {} \\
\text{(\# of vertex nodes)} - \text{(\# of bottom nodes)} \Bigr) .
\end{multline}
\end{defn}

Essentially the same proof as that of Lemma~\ref{dkr-lemma} gives us
then the following, perhaps somewhat surprising result.  While
edge nodes have weight only $\nicefrac{1}{2}$, we still get the same
estimate.

\begin{lemma}\label{sharper-lemma}
Let $D$ be a Newton diagram in two variables with support $K$
of size $d$ and suppose that its support $K$ is maximal, that is,
contains all $(a,b)$ with $a+b < d$. Then
\begin{equation}
SC(D) \geq \frac{d+1}{2}.
\end{equation}
\end{lemma}

\begin{proof}
As we said, the proof is similar to proof of Lemma~\ref{dkr-lemma}.
Again, we consider two adjacent levels in the diagram.
If the higher level has $n$ sign changes and the
lower level has $m$ sign changes then there must be at least
$\frac{\abs{m-n}}{2}$ nodes involving points only on these two levels.
This weighted count now includes half nodes at the edges.

What is left is to count the edge nodes that only involve the top level.
There are always 2 vertex nodes, plus if there $m$ sign changes at
the top level, then there are $d-m-1$ edge nodes.  An induction argument
completes the proof.
\end{proof}

The result is surprising because even though the nodes on the edges
and at the vertices have weight only half,
we still obtain the same estimate as in
Lemma~\ref{dkr-lemma}, except for the obvious two times $\nicefrac{1}{2}$
that we
lose for the two vertex nodes on the top level and the $1$ that we lose for
the bottom node. Thus Lemma~\ref{dkr-lemma} follows immediately from
Lemma~\ref{sharper-lemma}.  Do note that we now allow more half nodes at
the edge.  Therefore for a diagram with $d$ edge nodes and 3 vertex nodes
we again obtain $SC(D) = \frac{d+1}{2}$.

The result in Lemma~\ref{sharper-lemma} is
necessarily also sharp but we even have this stronger sharpness result that
will be useful later:

\begin{lemma}\label{prescribed}
Let $D$ be a maximal Newton diagram with support of size $d$.
Then we can find a
Newton diagram $\tilde{D}$ with support of size $d$
that is maximal, agrees with $D$
for all $(a,b)$ with $a+b = d-1$ and such that $SC(D)$ is exactly equal to
$\frac{d+1}{2}$.
\end{lemma}

\begin{proof}
Construct the example level by level, starting with the top level,
which is already
determined by $D$. Each next lower level is obtained by removing the leftmost
entry of the previous level, and then changing all the signs (so $P$ becomes
$N$ and vice verse). Using this construction there cannot be any interior
nodes, and the number of edge nodes on the sides is equal to the number of
sign changes on the top level. The number of edge nodes on the top level
is equal
to $d-1$ minus the number of sign changes. Adding these two numbers plus the
two vertex nodes on the top level gives $d+1$ nodes that each have weight 
$\nicefrac{1}{2}$.
\end{proof}

In the three dimensional proof we will perform several different operations
on $3$-dimensional Newton diagrams that change the shape of one or more
faces. We end the section with three $2$-dimensional lemmas that will be
needed to estimate how the number of nodes on the faces change with these
operations.

\begin{lemma}\label{SC-filling}
Let $D$ be a two-dimensional Newton diagram, and suppose that its support $K$
contains all points $(a,b)$ with $a+b < k$ and at least one point $(a,b)$
with $a+b = k$. Then there exist a Newton diagram $\tilde{D}$ that is an
extension of $D$ (i.e.\@ the $P$- and $N$-nodes remain unchanged) and
the support $\tilde{K}$ of $\tilde{D}$ contains
all points $(a,b)$ with $a+b = k$ and satisfies
\begin{equation}
SC(\tilde{D}) \leq SC(D).
\end{equation}
Finally the set $\tilde{K} \setminus K$ only contains points $(a,b)$
with $a+b=k$.
\end{lemma}

\begin{proof}
The proof is similar to that of Lemma~\ref{filling} except that we need to be
slightly more careful because different kinds of nodes receive different
weights. This time we change only one $0$-point at the time, starting with a
$0$-point that lies to the left or to the right of a $P$- or $N$-point. By assigning to this point a
$P$ or an $N$, there are a number of nodes of $D$ that have weight
$\nicefrac{1}{2}$ but that by the change either receive weight $1$ or are no
longer a node, depending on the sign we assign to the $0$-point. The situation
for these nodes is completely symmetric: if assigning a $P$ adds $c$ to the
weighted surface count then assigning an $N$ subtracts $c$ to the weighted surface count, where
$c$ is a multiple of $\nicefrac{1}{2}$. There is also at most one bottom node of
$D$ that has no weight, but that can turn into an edge node of $\tilde{D}$,
adding $\nicefrac{1}{2}$ a node to the weighted surface count depending on the choice of
sign. Now it is easy to see what can be done:

If $c=0$ then we can choose the sign such that no bottom node is changed into
an edge node. If $c \neq 0$ then we choose the sign such that the weighted surface
count is decreased by at least one half from the old $\nicefrac{1}{2}$-nodes, and
since the gain from the potential bottom node that changes into an edge node
is at most $\nicefrac{1}{2}$, the weighted surface count cannot increase.
\end{proof}

\begin{lemma}\label{sliceline}
Let $D$ be a Newton diagram with support $K$.  Suppose that
for some $k$, $(a,b) \in K$ whenever
$a+b < k$ and $(0,b), (1, b-1) \notin K$ when $b \geq k$. Then define
$K^\prime \subset \Z^2$ by
\begin{equation}
K^\prime = \{(a,b) \in \Z^2 \mid (a, b+1) \in K\}.
\end{equation}
Then there exists a Newton diagram $D^\prime$ with support $K^\prime$ such that
\begin{equation}
SC(D^\prime) \leq SC(D) - \frac{1}{2}.
\end{equation}
\end{lemma}

\begin{proof}
We can define $D^\prime$ naturally by assigning
$D'(a,b) = D(a,b+1)$.
The estimate follows immediately.
\end{proof}

\begin{lemma}\label{trianglegluing}
Let $D$ be a Newton diagram with support $K$ such that for some $k$,
$(a,b) \notin K$
whenever $a+b < k$.  Suppose that $K$ contains at least one point $(a,b)$ with
$a+b = k$. Define
\begin{equation}
K^\prime = K \cup \{(a,b) \in \N_0^2 \mid a+b < k\}.
\end{equation}
Then there is a Newton diagram $D^\prime$ with support $K^\prime$ such that
$SC(D^\prime) \leq \frac{k}{2} + SC(D)$.
\end{lemma}

\begin{proof}
By Lemma~\ref{prescribed} we know that given any choices for the level $\{(a,b)
\mid a+b = k-1\}$, there exist a Newton diagram $E$ of degree $k-1$ with the
chosen top level and maximal $K$ such that the $SC(E) = \frac{k+1}{2}$. We can
choose the top level of $E$ such that by merging $E$ with $D$ no new interior
nodes are created, and such that at least one top level edge node of $E$ is
canceled by a point of $D$ at level $k$. If $D^\prime$ is the Newton diagram
obtained by merging $E$ with $D$ then it follows that
\begin{equation}
SC(D^\prime) \leq \frac{k+1}{2} + SC(D) -
\frac{1}{2}.
\end{equation}
\end{proof}

%%%%%%%%%%%%%%%%%%%%%%%%%%%%%%%%%%%%%%%%%%%%%%%%%%%%%%%%%%%%%%%%%%%%%%%%%%%%%

\section{Three-dimensional results}

Before we start proving our main result, where we use two-dimensional
estimates to do a weighted count of the number of nodes that must occur on each of the faces
of a Newton diagram, we should first carefully define what we mean by a face.

\begin{defn}
A {\emph{face point}} of a Newton diagram $D$ with support $K$ is a point of $K$
that is adjacent to a $0$-point. A {\emph{vertical face}} of $K$ is a
maximal connected family of face points $a = (a_1, \ldots, a_n)$ that lie in a
plane $\{a_j = C\}$. A {\emph{horizontal face}} is a maximal connected family
of face points that lie in a plane $\{\abs{a} = C\}$.
\end{defn}

We can think of a simplex centered at each point of $K$.  Then the union of
the simplices form a polyhedron.  Then faces of the polyhedron correspond
to the faces of $K$ as defined above.  See
Figure~\ref{fig:newtonface} where a vertical face is shaded.

\begin{figure}[h!t]
\begin{center}
\includegraphics{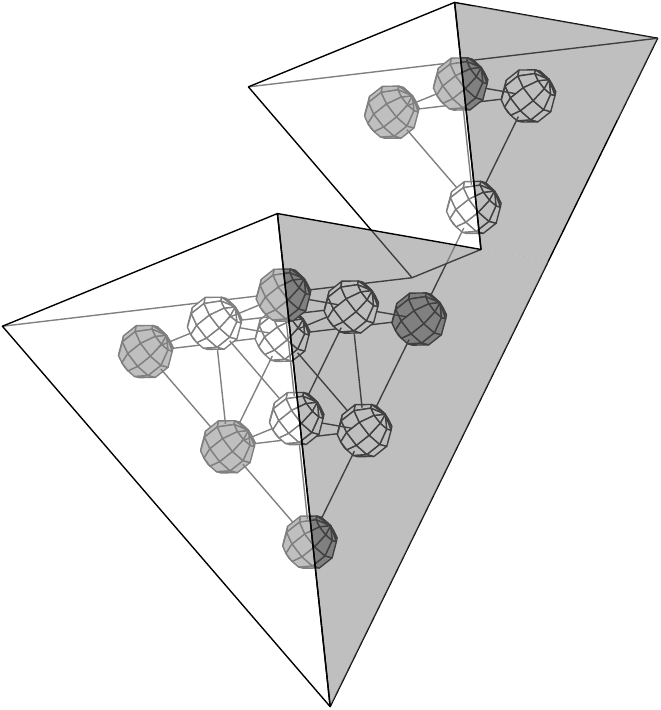}
\caption{Face in a Newton diagram.\label{fig:newtonface}}
\end{center}
\end{figure}

For each face $F$ of $K$ we define $\#(F)$ as the minimal number of nodes
that must occur for any possible configuration $F \rightarrow \{0, P, N\}$,
but counting only nodes that involve only points of $F$ and points that do not
lie in $K$. In other words, if it is possible for a node in $F$ to be
cancelled by $N$- or $P$-points outside the face $F$, we assume that they are
cancelled.

As before, we write $SC(F)$ for the number of nodes of $F$ (as above) but
weighting edge nodes as $\nicefrac{1}{2}$, weighting vertex nodes
in vertical faces
as $\nicefrac{1}{2}$ except for bottom nodes, and not counting bottom nodes
in all faces and vertex nodes in horizontal faces.
We get the following inequality:
\begin{equation}
SC(D) := \cup_{F} SC(F) \leq \#(D).
\end{equation}
Where $F$ ranges over all the faces of $K$.
Since bottom nodes are not counted for $SC(D)$ we obtain that
\begin{equation}
SC(D) \leq \#(D) - \text{(\# of bottom nodes)}.
\end{equation}

We prove the following result:

\begin{thm}
Let $D$ be a 3-dimensional Newton diagram with support $K$
of size $d$, and
assume that $K$ has no overhang.  Then $D$ has at least $2d+2$ nodes.
\end{thm}

\begin{proof}
We will actually prove a stronger statement. Instead of estimating the number
of nodes of $D$, we estimate for each face of $K$ the number of nodes that
must at least occur on that face of $K$, and add those numbers.  Of course by
doing so a node that occurs on more than one face will be counted more than
once.

A node that occurs at an edge (which means two of the involved points must be
$0$) will be counted twice, and a node that occurs at a vertex (which means
three of the involved points are $0$) is counted three times. In order to get
the right count we will give weight $\nicefrac{1}{2}$ to the edge nodes for each face.
Vertex nodes are only counted (with weight $\nicefrac{1}{2}$) on vertical
faces, and only when they are not bottom nodes (otherwise they have weight
zero).

\begin{defn}
Let $K$ be the support of a $3$ dimensional Newton diagram $D$, and let $K_1,
\ldots, K_n$ be the two dimensional faces of $K$. For a vertical face $K_j$ we
define
\begin{equation}
SC(K_j) = \min_{D_j} SC(D_j),
\end{equation}
where the minimum is taken over all $2$-dimensional Newton diagrams $D_j$
with support $K_j$. For a horizontal face we use the same definition except that
none of the vertex nodes are counted at all (all have weight zero), so then
\begin{equation}
SC(D_j) = \text{(\# of interior nodes)} + \frac{1}{2} \times \text{(\# of edge  nodes)}.
\end{equation}
Finally we define
\begin{equation}
SC(K) = \sum SC(K_j).
\end{equation}
\end{defn}

Of course the number we obtain is a lower estimate and does not have to be equal to the actual number of
nodes of $D$ for several reasons. First, interior nodes are not accounted for. Second, the $2$-dimensional Newton diagrams
$D_j$ do not have to correspond to the diagrams on the faces of $D$. In fact,
we don't require that the Newton diagrams $D_j$ ``match'' at all, so there may not
be any Newton diagram $D$ that has optimal $D_j$ as faces. Finally, bottom
nodes all have weight 0, so for each bottom node we count one node too
few. Nevertheless we will be able show that the sum is at least $2d+1$, and
since there is at least one bottom node (the node involving the point
$(0,0,0)$), this will conclude the proof of the theorem.

Before giving the specifics, we note that we can assume that the p-degree of
a polynomial corresponding to $D$ is $d$.  That is, we could translate $K$
to assume that for each $j = 1,2,3$, there is a point $m \in K$
such that $m_j = 0$, but still assume that $K \subset \N_0^3$.  Now by
requiring that $K$ has no overhang, it is not hard to see that $K$ must
contain the point $(0,0,0)$.

For the purpose of contradiction we assume that there exists a $K$ with
$SC(K) < 2d+1$. Let the $d$ be the smallest size for which such $K$
exists.  Further assume that for this $d$ the number of points of $K$,
$\abs{K}$ is maximal, i.e.\@ the number of $P$s and $N$s of $D$ are
maximal.

Now look at the largest $k$ so that $(a,b,c) \in K$ whenever $a+b+c < k$.
Either $k = d$ or $k < d$.

If $k = d$ then we can estimate the number of surface nodes of $K$ by
applying the estimates obtained in \S~\ref{twodimensions} to each of the
sides of the tetrahedron $K$. For each of the three vertical faces we obtain
the estimate $\frac{d+1}{2}$ by Lemma~\ref{sharper-lemma} and for the
horizontal face only $\frac{d-1}{2}$ since the three vertex nodes have weight zero. In total we obtain that $SC(K) \geq 2d+1$, which contradicts our
assumption.

Therefore we may assume that $k < d$. Now we look at the plane
$S_k = \{(a, b, c) \in \N_0^3 \mid a+b+c = k\}$.
Then there are two possibilities. Either $K \cap S_k$
contains at least one element on each of the edges
$\{(0, b, c)\}$, $\{(a, 0, c)\}$ and $\{(a, b, 0)\}$, or $K \cap S_k$ does not
contain an element on one of these edges. Let us deal with the latter case
first.

\textbf{One of the edges is empty: Slice off a face.}

In this case we may without loss of generality assume that $K \cap S_k$ does
not contain elements of the form $(0, b, c)$. Then it follows from the
{\emph{no overhang}} condition that $K$ does not contain any elements of the
form $(0, b, c)$ with $b+c \geq k$. Therefore if we remove from $K$ all
elements of the form $(0, b, c)$ and we subtract $1$ from the first
coordinate of all remaining elements of $K$, we obtain a new subset
$K^\prime \subset \N_0^3$ that satisfies the hypotheses of the theorem, except
that it now has size $d-1$.
See Figure~\ref{fig:newtonslice}.
By the assumption on the minimality of $d$ we have
that the number of surface nodes of $K^\prime$ must be at least $2d-1$. The
contradiction therefore follows if we show that $SC(K^\prime)$ is at least $2$
smaller than $SC(K)$.

\begin{figure}[h!t]
\begin{center}
\includegraphics{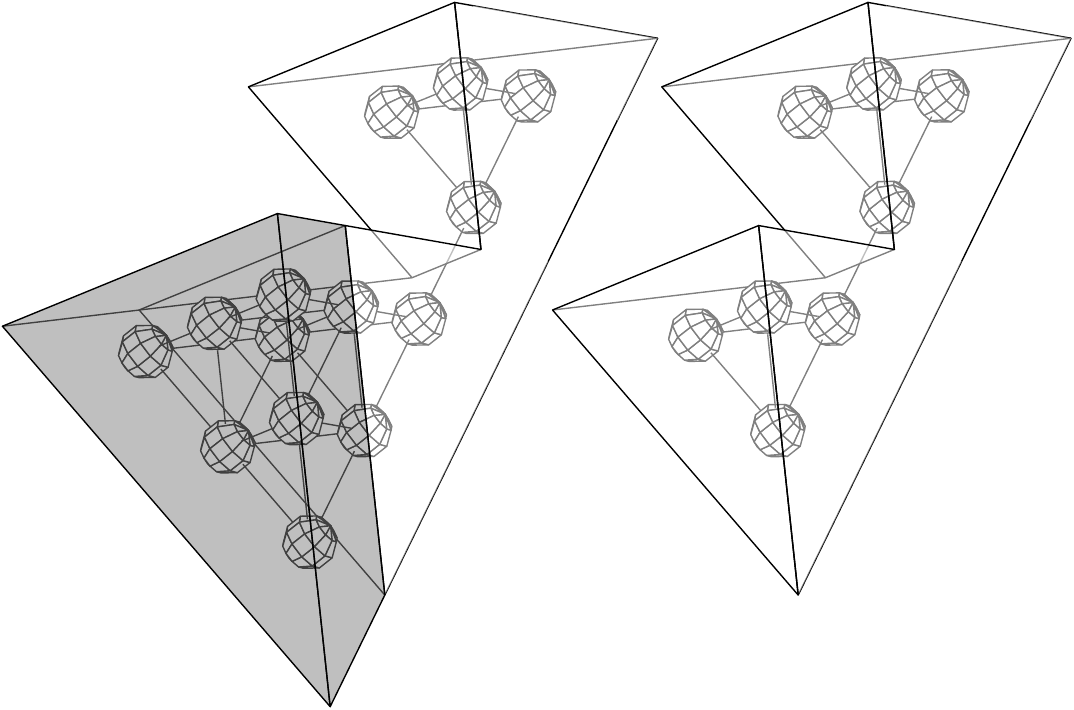}
\caption{Slicing off a face.\label{fig:newtonslice}}
\end{center}
\end{figure}

Fortunately almost all the faces of $K^\prime$ are equal to those of $K$.
There are two side faces that lose one line, namely the faces in the $(b = 0)$-
and  $(c=0)$-planes. Both of these faces lose $\nicefrac{1}{2}$
a surface node by
Lemma~\ref{sliceline}. The horizontal face in the $S_{k-1}$-plane also loses
one line but it may or may not lose half a node. If the intersection $K \cap
S_k \cap \{(1, b, c)\}$ is empty then it is clear that the horizontal face
does lose half a node.  To see this, we refer back to the proof of
Lemma~\ref{sharper-lemma}.  If the intersection is empty, we have removed a
full line of points from the face, thus removing at least half a node.
However, if this intersection is
non-empty then the horizontal face may not lose surface nodes (keep in mind
that vertex nodes have weight zero for the horizontal faces).

One direction remains to be considered. If $K \cap S_k \cap \{(1, b, c)\}$ is
empty then the face $K \cap \{(0, b, c)\}$ is replaced by the face $K^\prime
\cap \{ (0, b, c) \}$, which is just a triangle with side length one smaller.
Therefore the weighted surface count for this smaller triangle is one-half smaller and
the estimate is complete. When $K \cap S_k \cap \{(1, b, c)\}$ is not empty
the situation is slightly more complicated. Again we lose the surface nodes
on the face $K \cap \{(0, b, c)\}$, but we gain nodes when the faces of $K$
in the plane $\{(1,b,c)\}$ are glued to the triangle $\{(1,b,c) \mid b+c <
k-1\}$.  In this case it follows from Lemma~\ref{trianglegluing} that we lose
at least $1$ node which completes the estimate.

\textbf{None of the edges is empty: Fill $S_k$.}

The only case that is left to deal with is when $K \cap S_k$ contains at
least one element on each of the edges $\{(0, b, c)\}$, $\{(a, 0, c)\}$ and
$\{(a, b, 0)\}$. The idea now is to take enlarge $K$ by defining $K^\prime = K
\cup S_k$.  See Figure~\ref{fig:newtonfill}.
By our assumption on $k$ we have that $\abs{K^\prime}$ is strictly
larger than $\abs{K}$, and $K^\prime$ satisfies the no-overhang condition as
well.  If we can show that the number of surface nodes of $K^\prime$
is not larger than that of $K$, then we get a contradiction with our
assumption that $\abs{K}$ is maximal.

\begin{figure}[h!t]
\begin{center}
\includegraphics{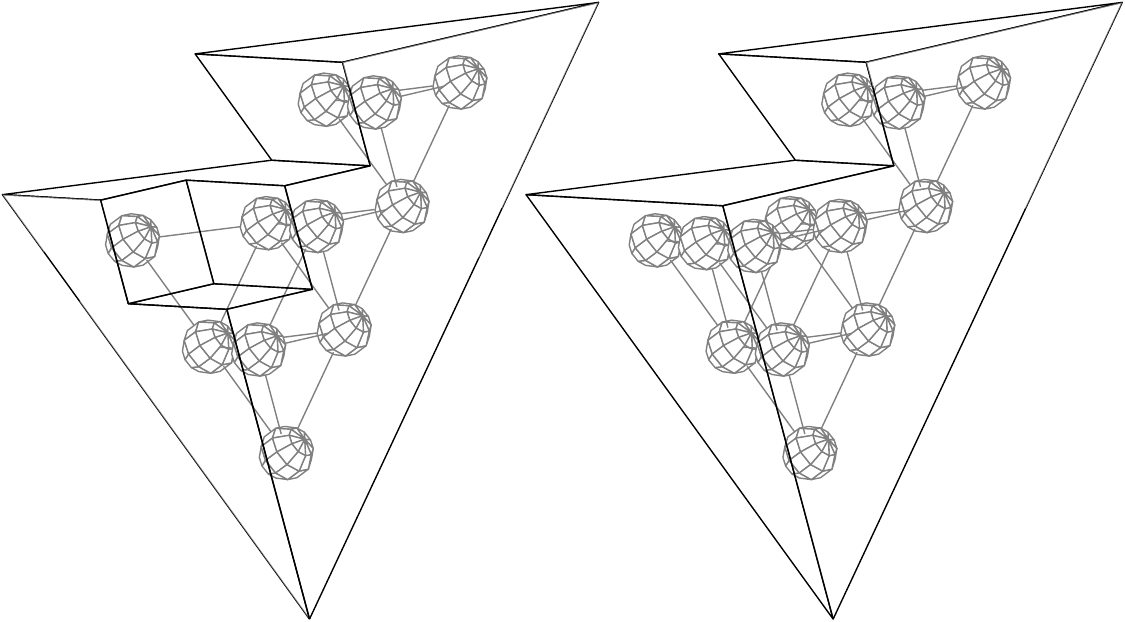}
\caption{Filling a 3-dimensional Newton diagram.\label{fig:newtonfill}}
\end{center}
\end{figure}

We apply Lemma~\ref{SC-filling} to see that
the number of surface nodes on each
of the side faces $\{a=0\}, \{b=0\}$, and $\{c=0\}$ does not increase when we
extend $K$ to $K^\prime$.
All the other vertical faces can only become smaller when we replace $K$ by
$K^\prime$, so the number of surface nodes on those faces only become
smaller.

The only faces that are left are the horizontal faces on the levels $k-1$ and
$k$. Let us say that we start with some optimal $2$-dimensional Newton
diagrams $D_j$ with supports $K_j$ corresponding to the horizontal faces at
levels $k-1$ and $k$ for set $K$. We create a configuration for the
horizontal face of $K^\prime$ at level $k$ as follows.

We keep the same Newton diagrams for the faces at level $k$ that we already
had. We replace every $0$ at level $k$ of $K$ and coordinate $(a,b,c)$, where
$a \neq 0$, by the sign of the point $(a-1, b, c)$ of the Newton diagram $D_j$
containing that point, although it is possible that we need to reverse all the
signs for a face $D_j$, as we shall see below. After we have done this
procedure, we will see that we can assign signs to the
remaining $0$-points of the form $(0, b, c)$ with $b+c = k$ in a way that does
not increase the number of nodes.

By sliding the faces at level $k$ upwards in this way there can
only be fewer internal nodes, as the faces can become smaller but not larger.
There are vertex (or edge) nodes that remain vertex (or edge) nodes so we do
not need to worry about those. However, when `gluing' these configurations to
the configurations of the old faces at level $k$ it is also possible for edge
and vertex nodes to become internal nodes, which have weight one.
This is possible in two different ways.

The first way is for an vertex node of an old $k$-level face to become an
internal node. Since the horizontal vertex had weight zero, this gives us one
extra node, but as we remarked earlier, we lose twice half a node on vertical
faces when such a vertex becomes internal, so we do not need to worry about
this case.

The other way is for an edge node of an old $k$-level face to become an
internal node. We can deal with this case by flipping the signs of old level
$k$ configurations. By doing this we get that at most half of the edge nodes
of old configurations that actually become internal nodes,
and the other edge nodes disappear, which solves this case as well.

That leaves us with the remaining $0$-points on the edge $\{(0, b, c)\mid b+c
= k\}$. As we only have to worry about the horizontal faces (as the weighted count for
the vertical faces is already taken care of) and vertex-nodes in horizontal
faces have weight zero, it is easy to assign $P$ or $N$ to each point one by
one, starting with points adjacent to non-$0$ points in the edge (which we
assumed exist), in such a way that no interior nodes are created and the
number of edge nodes remains the same. This completes the proof.
\end{proof}

%%%%%%%%%%%%%%%%%%%%%%%%%%%%%%%%%%%%%%%%%%%%%%%%%%%%%%%%%%%%%%%%%%%%%%%%%%%%%

\section{Sharp polynomials in three dimensions} \label{section:sharppols}

In the previous section we have obtained some information of what a Newton
diagram with minimal number of nodes looks like.  While it is not hard to construct a polynomial $Q$ corresponding to a given Newton diagram, in
general we will get many more nonzero terms in the corresponding $P \in
\sI(3)$ than there are nodes in the diagram.

We will say an indecomposable polynomial $P \in \sI(n,d)$ is \emph{sharp} if
$N(P)$ is minimal among all indecomposable polynomials $\sI(n,d)$.
We make a similar
definition for $\sH(n,d)$;
$p \in \sH(n,d)$ is \emph{sharp} if
$N(p)$ is minimal in $\sH(n,d)$.
We will concentrate on $\sH(n,d)$ in this
section since we have proved sharp bounds only for $\sH(3,d)$
and not for $\sI(3,d)$.  Also using the nonhomogeneous notation will make
the computations and the exposition simpler.  Of course, if it turns out
that indecomposable polynomials in $\sI(3,d)$ posses the same degree bounds
as $\sH(3,d)$, then the sharp polynomials we describe in this section
will also be sharp in $\sI(3,d)$.

It is an easier problem
to list all possible Newton diagrams with the sharp
number of nodes, but it is harder to decide which diagrams correspond to
actual sharp polynomials in $\sI(n,d)$ or $\sH(n,d)$.
For $n=2$ the classification of sharp polynomials is not trivial even in
$\sH(2,d)$, see~\cites{DL:complex,LL}.
We will see that classification in dimension 2
is related to classifying the sharp
polynomials in dimension 3.

It is not hard to construct certain sharp polynomials $p \in \sH(3,d)$.
That is, $p$ achieves equality in the bound
\begin{equation} \label{sharpn3:eq}
\deg(p) \leq \frac{N(p)-1}{2} .
\end{equation}
Start with
\begin{equation}
p_1(x) := s(x) = x_1 + x_2 + x_3 .
\end{equation}
Now construct
\begin{equation}
p_d(x) := p_{d-1}(x) - x_3^{d-1} + x_3^{d-1} ( s(x) - 1 ) .
\end{equation}
It is not hard to see that $p_d \in \sH(3,d)$ and that $N(p_d) = N(p_{d-1})+2$,
and therefore $N(p_d) = 2d+1$.  Thus, $p_d$ are sharp.
These polynomials are known (see for
example \cite{DLP}) as \emph{generalized Whitney polynomials}.
At each step, instead of
$x_3^{d-1}$ we can also pick an arbitrary monomial of $p_{d-1}$ of degree $d-1$.
We then obtain all the other generalized Whitney polynomials.

The Newton diagram for the Whitney polynomials is easy to see.
In the nonhomogeneous setup, where we just take
the Newton diagram of $q = \frac{p-1}{s-1}$,
the diagram is
simply a connected path of $d$ $P$s
going from $(0,0,0)$ to $(a,b,c)$ where $a+b+c = {d-1}$.  In the homogeneous
setup, the signs are not all $P$, but instead alternate between $P$ and $N$.

It is natural to ask what all the sharp polynomials look like. In
particular, one might wonder if there are there any examples
that are not generalized Whitney polynomials.  The following
construction was given in \cite{DLP}.  Start with
\begin{equation} \label{faran:eq}
x_1^3+3x_1x_2+x_2^3,
\end{equation}
which is sharp and is in $\sH(2,3)$.  Now replace $x_2$ with $x_2+x_3$ to obtain
\begin{equation} \label{faran3:eq}
\begin{split}
p(x)
&:=
x_1^3+3x_1(x_2+x_3)+(x_2+x_3)^3 \\
&\phantom{:}=
x_1^3+3x_1x_2+3x_1x_3+x_2^3 +3x_2^2x_3 + 3x_2x_3^2 + x_3^3
.
\end{split}
\end{equation}
The $p$ obtained is in $\sH(3,3)$, and furthermore it has $2d+1 = 7$ terms
and hence is sharp.  The obtained Newton diagram has maximal support,
that is, the support $K$ of the diagram contains all points $m$
such that $\abs{m} < d$.
For degrees up to 7 we can prove the following proposition.

\begin{prop} \label{filledsharp:prop}
Suppose that $p \in \sH(3,d)$ is sharp and that the induced Newton diagram
has maximal support $K$ (that is, $K$ contains all points
$m \in \N_0^3$ with $\abs{m} < d$).
Also assume that $d \leq 7$.  Then up to
permutation of variables, $p$ is equal to either $s$ or \eqref{faran3:eq}.
\end{prop}

We conjecture that this proposition holds for all degrees.  Before we prove
the proposition, we will need the following lemma.  See \cite{DLP} and
\cite{DKR} for proof and more discussion of these results.

\begin{lemma} \label{undoing:lemma}
Let $p \in \sH(n,d)$.
\begin{enumerate}[(i)]
\item If $p$ is homogeneous of degree $d$, then $p = s^d$.
\item Write
\begin{equation}
p(x) = \sum_{\abs{\alpha} = d} c_\alpha x^\alpha
+
\sum_{\abs{\alpha} < d} c_\alpha x^\alpha ,
\end{equation}
then we have
\begin{equation}
p(x) = s^d
+
\sum_{\abs{\alpha} < d}
c_\alpha x^\alpha
\left( 1-s^{d-\abs{\alpha}} \right) .
\end{equation}
\end{enumerate}
\end{lemma}

\begin{proof}[Proof of Proposition~\ref{filledsharp:prop}]
First homogenize $p$, then change variables to obtain
a polynomial $P$ in $\sI(3)$.  By setting each variable to zero
in turn, we obtain four 2-dimensional diagrams with maximal support for
polynomials in $\sI(2)$.

Let $D_1, \ldots, D_4$ be the 2-dimensional diagrams obtained above.
We notice that $N(P)$ is bounded below by the number
$SC(D_1)+SC(D_2)+SC(D_3)+SC(D_4)+1$.  Now each
$SC(D_i)$ is bounded below by $\frac{d+1}{2}$ using
Lemma~\ref{sharper-lemma}.

The following statements are easy to verify for sharp polynomials.
If $p \in \sH(3,d)$ is sharp and the corresponding Newton diagram
has maximal support, then:
\begin{enumerate}[(i)]
\item \label{filledsharp:proof:i}
There exists no term in $p$ of degree strictly less than $d$ depending on all
3 variables.
\item \label{filledsharp:proof:ii}
For each $i=1,2,3$, there is precisely one pure term $x_i^d$.
\item \label{filledsharp:proof:iii}
For each $i=1,2,3$, the number of non-pure terms not depending
on $x_i$ is bounded below by $\frac{d-1}{2}$ if we weight terms of
degree $d$ as $\nicefrac{1}{2}$ a term and terms of degree less than $d$
have weight 1.
\item \label{filledsharp:proof:iv}
The number of non-pure degree-$d$ terms is bounded below by
$\frac{d-1}{2}$ if we weight terms that depend on only 2 variables for
$\nicefrac{1}{2}$ a term and the terms that depend on all 3 variables for 1 term.
\end{enumerate}
The reason for weighting some terms for $\nicefrac{1}{2}$ is, of course, because
we count them twice, once when counting the degree-$d$ terms and once when
counting terms not depending on a certain variable.
The item \eqref{filledsharp:proof:i}
follows because the polynomial is sharp, and $\sum_{j=1}^4 SC(D_j)$
is already the sharp number.
Item \eqref{filledsharp:proof:ii} was noticed by
\cite{DKR}.  It follows because a diagram for
a polynomial in $\sH(2,d)$ that has maximal support
can only contain one node for a
pure term degree $d$.  If there were any other node for a pure term,
it would force a negative pure term in $p$.

Now we use Lemma~\ref{undoing:lemma}.
We start with $s^d=(x_1+x_2+x_3)^d$.  We are allowed to
subtract from $s^d$ polynomials of the form
\begin{equation} \label{eq:subtractions}
c_\alpha x^\alpha
\left( 1-s^{d-\abs{\alpha}} \right) .
\end{equation}
where $x^\alpha$ does not depend on a single variable and is of degree
$d-1$ or less.  We are only allowed to subtract polynomials of
the form \eqref{eq:subtractions} such that we do not
violate the bounds given above.
In particular, we are only allowed to subtract \eqref{eq:subtractions}
for terms $x^\alpha$
where $x^\alpha$ does not depend on all variables; that is, it is
independent of at least one variable.
For example, for $d=7$ we are allowed to do
up to 3 such subtractions for each missing variable (so altogether up to 9
subtractions).
We must ensure that the subtractions
also remove enough terms of degree $d$.

It is not hard but it is tedious to check all the possibilities for low
degrees.  We have verified the claim up to $d \leq 7$.  We leave the details
to the reader.
\end{proof}

Proposition~\ref{filledsharp:prop} together with the geometric insight from
the proof of the
main result seem to suggest to the authors that the following construction
generates all sharp $p \in \sH(3,d)$.  Start
with $p$ being either $s$ or the polynomial \eqref{faran3:eq},
possibly with permuted variables.
Then construct others inductively.
Pick a sharp $p \in \sH(3,d)$, pick a monomial $m$ in $p$ of degree $d$.
Let $g$ be either $s$ or the polynomial \eqref{faran3:eq} (possibly with
permuted variables).
Construct $\tilde{p}$ as
\begin{equation}
\tilde{p} := p-m+mg.
\end{equation}
It is easy to see that $\tilde{p} \in \sH(3,d+k)$, where $k=1$ if $g=s$
or $k=3$ if $g$ was \eqref{faran3:eq}.  Furthermore,
$\deg{\tilde{p}} = \frac{N-1}{2}$, so $\tilde{p}$ is sharp.

%%%%%%%%%%%%%%%%%%%%%%%%%%%%%%%%%%%%%%%%%%%%%%%%%%%%%%%%%%%%%%%%%%%%%%%%%%%%%

\section{Higher dimensions} \label{section:higherdim}

It would have been welcome if the methods that we have developed
so far were applicable in higher dimensions. Unfortunately the combinatorics in dimensions $4$ and higher is considerably different. The proof
of the main estimate cannot possibly work for the following reason.  We prove
the main sharp estimate by noting that we can in some
sense fill the whole diagram.  Such a technique must always fail in higher
dimensions.  We make the following observation.

\medskip

\emph{For $n \geq 4$, a Newton diagram $D$ with support $K$ of
size $d \geq 2$
such that $K$ contains all $m \in \N_0^n$ with $\abs{m} < k$ has too many nodes.}

\medskip

Here \emph{too many} means more than $d(n-1)+2$ nodes.  There exist
indecomposable polynomials in $\sI(n,d)$ with precisely $d(n-1)+2$ terms.
These are induced by
the generalized Whitney polynomials as described in the previous section
or in~\cite{DLP}.  These polynomials are in $\sH(n,d)$ and have
$d(n-1)+1$ positive terms.
Therefore, the sharp estimate for $n \geq 2$,
cannot be obtained by looking at a completely filled Newton diagram.

Next we verify this observation.  We will only count nodes
that appear on two-dimensional faces of the filled-in diagram.
It also suffices to only consider $n=4$.  For $n=4$, the
support has 10 two-dimensional faces.  Let $D$ be the completely filled
diagram in 4 dimensions.  Let us not count the vertex nodes
as there are exactly 5 of those.  Let us write $D_f$ for a diagram of a
face.  Each edge node of $D_f$ contributes $\nicefrac{1}{3}$ to the total
weighted count
of $\#(D)$, as we count each edge 3 times.  Each face node contributes $1$
to $\#(D)$.  If we look at the proof of Lemma~\ref{sharper-lemma}, we notice
that
\begin{equation}
\frac{1}{3} \times \text{(\# of edge nodes)}
+
\text{(\# of face nodes)} \geq
\frac{d-1}{3} .
\end{equation}
The way to minimize this number is to only have edge nodes.  We
then must have $d-1$ edge nodes.  Hence for $d > 1$ we have
\begin{equation}
\#(D) \geq 10 \left( \frac{d-1}{3} \right) + 5 > 3d+2 .
\end{equation}
As noted above,
the generalized Whitney polynomials constructed in \cite{DLP} and the
previous section for $n=4$ give diagrams with exactly $3d+2$ nodes.

In particular, no example such as \eqref{faran3:eq} exists when $n > 3$.
In \cite{DLP} it was proved that for $n$ sufficiently large compared
to the degree $d$,
all sharp polynomials in $\sH(n,d)$ are generalized Whitney polynomials.
With
the geometric insight and the techniques developed in this paper, the
authors conjecture that all sharp polynomials in $\sI(n)$,
$n \geq 4$ are Whitney.

While we are not able to prove the sharp bound for larger dimensions,
we can use the $n=2$ results to prove a degree bound (alas not sharp)
for indecomposable polynomials in $\sI(n)$ for all $n \geq 3$ without the
positivity requirement.
We will follow the pullback
procedure from \cite{DLP} to prove the following bound.

\begin{thm} \label{thm:eekbound1}
Let $P \in \sI(n)$, $n \geq 3$ be indecomposable and suppose that $P$ contains
one pure monomial.  Then
\begin{equation}
\operatorname{p-degree}(P) \leq
\frac{2n(2N(P)-5)}{3n^2-3n-2} \leq \frac{4}{3}~\frac{2N(P)-5}{2n-3}.
\end{equation}
\end{thm}

The bound is slightly different from \cite{DLP} and \eqref{highnboundirred}
only because in the projectivized version we count one extra term.
The following proof and computation is essentially the same as in \cite{DLP}
but is adapted to our setting and slightly simplified.

\begin{proof}
Let $P \in \sI(n)$, $n \geq 3$ be indecomposable and suppose that
$P$ contains the monomial $X_n^d$.  That is we can write
\begin{equation}
P(X) = (X_0+X_1+X_2+\cdots+X_n)Q(X) .
\end{equation}
Let $D=2n-3$.
Now take the sharp degree $D$ polynomial mapping $\phi$ from~\cite{DKR},
which has the homogenized form
\begin{equation}
(u,v,t) \mapsto (u^D, v^D, c_1 u^{D-2}vt, c_2 u^{D-4} v^2t^2,\ldots,t^D ) .
\end{equation}
The precise values of $c_j$ are not important for us.
This map has exactly $n+1$ components and furthermore
\begin{equation}
u^D + v^D + c_1 u^{D-2}vt + c_2 u^{D-4} v^2t^2 + \cdots -t^D
=
(u+v-t)q(u,v,t) .
\end{equation}
If we swap $t$ for $-t$ we obtain a map $\tilde{\phi}$
\begin{equation}
(u,v,t) \mapsto (u^D, v^D, -c_1 u^{D-2}vt, c_2 u^{D-4} v^2t^2,\ldots,t^D )
\end{equation}
that is
\begin{equation}
u^D + v^D - c_1 u^{D-2}vt + c_2 u^{D-4} v^2t^2 - \cdots +(-1)^{D+1}t^D
=
(u+v+t)\tilde{q}(u,v,t) .
\end{equation}
We modify $P$ to create $\tilde{P}$
by replacing $X_2$ with $-X_2$, $X_4$ with $-X_4$ and so on
\begin{equation}
\tilde{P}(X) := (X_0+X_1-X_2+\cdots+(-1)^{D+1}X_n)\tilde{Q}(X) .
\end{equation}
The signature of
$(X_0+X_1-X_2+\cdots+(-1)^{D+1}X_n)$ is exactly the same as the signature
of $\tilde{\phi}$.  Furthermore $N(\tilde{P}) = N(P)$.  We can now
compose $\tilde{P} \circ \tilde{\phi}$ and note that
\begin{equation}
(\tilde{P} \circ \tilde{\phi})(u,v,t) = (u+v+t)q_1(u,v,t) .
\end{equation}
Furthermore, it is easy to see that $N(\tilde{P} \circ \tilde{\phi}) \leq
N(\tilde{P}) = N(P)$.  It is also not hard to see that
$\tilde{P} \circ \tilde{\phi}$ is indecomposable since $\tilde{P}$ and
$\tilde{\varphi}$ are indecomposable.  This fact can be seen by either looking
at the Newton diagram, or simply by attempting to write
$\tilde{P} \circ \tilde{\phi}$ as a sum of two polynomials in $\sI(2)$
with distinct monomials.  What remains is to estimate the
p-degree of
$\tilde{P} \circ \tilde{\phi}$ and apply the 2-dimensional result.

Suppose that
the p-degree of $P$ is $d$.  We can divide through by any common monomial
factors to
make sure that without loss of generality $P$ is of degree $d$.
Suppose that $P$ (and hence $\tilde{P}$) contains a constant
multiple of the monomial
$X_0^{a_0}\cdots X_{n-1}^{a_{n-1}}$.
Such a monomial must exist if
the p-degree of $P$ is $d$.  After composing with $\tilde{\phi}$
we see that we see that
$X_0^{a_0}\cdots X_{n-1}^{a_{n-1}}$ will have a factor
of $t^c$ where
\begin{equation}
\begin{split}
c & =
d - (a_0D + a_1 D + a_2(D-1) + \cdots + a_{n-1}(D-n+2)) \\
& =
d - \left( D \sum_{j=0}^{n-1} a_j - \sum_{j=2}^{n-1} (j-1)a_j \right) \\
& =
d - Dd + \sum_{j=2}^{n-1} (j-1)a_j .
\end{split}
\end{equation}
We have used that $\sum a_j = d$.
Since $P$ contains the term $X_n^d$ then
\begin{equation}
\operatorname{p-degree}(\tilde{P} \circ \tilde{\phi}) \geq d-c =
dD - \sum_{j=2}^{n-1} (j-1)a_j .
\end{equation}
We apply the 2-dimensional result to get
\begin{equation}
\operatorname{p-degree}(\tilde{P} \circ \tilde{\phi}) \leq
2N(\tilde{P} \circ \tilde{\phi}) - 5
\leq
2N(P) - 5 .
\end{equation}
We combine the inequalities (and use $D=2n-3$) to obtain
\begin{equation}
\begin{split}
\operatorname{p-degree}(P) = d
& \leq
\frac{2N(P) - 5 + \sum_{j=2}^{n-1} (j-1)a_j}{2n-3} \\
& \leq
\frac{2N(P) - 5}{2n-3} +
\frac{\sum_{j=2}^{n-1} (j-1)a_j}{2n-3}  .
\end{split}
\end{equation}
We can assume that $a_j$ are in decreasing order.  We estimate
\begin{equation}
\frac{\sum_{j=2}^{n-1} (j-1)a_j}{2n-3}
\leq \frac{d}{(2n-3)n} \sum_{j=2}^{n-1} (j-1) =
\frac{d}{(2n-3)n} \binom{n-1}{2} .
\end{equation}
The expression
$\frac{1}{(2n-3)n} \binom{n-1}{2}$ is easily seen to be strictly less than
1.  Therefore,
\begin{equation}
\begin{split}
d & \leq
\frac{2N(P) - 5}{2n-3} +
\frac{\sum_{j=2}^{n-1} (j-1)a_j}{2n-3}
\\
& \leq
\frac{2N(P) - 5}{2n-3} +
d
\left(\frac{1}{(2n-3)n} \binom{n-1}{2} \right) .
\end{split}
\end{equation}
After some computation we obtain
\begin{equation}
d
\leq
\frac{2n (2N(P)-5)}{3n^2 -3n -2} .
\end{equation}
To find a simpler although slightly weaker estimate we note that when $n \geq
2$ we have
$\frac{2n}{3n^2-3n-2} \leq \frac{4}{3(2n-3)}$ and therefore
\begin{equation}
d \leq \frac{4}{3} \frac{2N(P)-5}{2n-3} .
\end{equation}
\end{proof}

The theorem is enough to prove the inhomogeneous
bound \eqref{highnboundirred}
of Theorem~\ref{thm:nonhomog}.  After homogenizing the expression
$p(x)-1$, the $-1$ will contribute a pure term
required in Theorem~\ref{thm:eekbound1}.
Unfortunately, without assuming that we have a pure monomial
we cannot estimate the p-degree after composition.  We will prove
a much weaker bound using the 2-dimensional bound
in a simpler way.  However, this
bound will work for all indecomposable polynomials in $\sI(n)$.

\begin{thm} \label{thm:eekbound2}
Let $P \in \sI(n)$, $n \geq 2$ be indecomposable.
Then
\begin{equation}
\operatorname{p-degree}(P) \leq
(n-1)(2N(P)-5) .
\end{equation}
\end{thm}

\begin{proof}
Without loss of generality, suppose that the monomials of $P$ have no
common monomial divisor.  Then we note that there must exist monomials
of the form
\begin{equation}
X_1^{a_1}X_2^{a_2}\ldots X_{n}^{a_{n}}
\qquad \text{and} \qquad
X_0^{b_0}X_2^{b_2}\ldots X_{n}^{b_{n}} .
\end{equation}
We define
\begin{equation}
\tilde{P}(X_0,X_1,T) :=
P(X_0,X_1,X_2T,X_3T,\ldots,X_nT) .
\end{equation}
When we fix $X_2 + X_3 + \cdots + X_n = 1$, then $\tilde{P} \in \sI(2)$.
Define $X' = (X_2,\ldots,X_n)$ for ease of notation.
We want to pick a fixed $X' = X_2, \ldots, X_n$ such that for every
monomial $X_0^iX_1^j (X')^\alpha$ with a nonzero coefficient in
$Q = P/(X_0+\cdots+X_n)$, there exists a monomial
$X_0^iX_1^jT^{d-i-j}$ in $\tilde{Q} = \tilde{P}/(X_0+X_1+T)$ with a nonzero
coefficient as well.

It is not hard to see that such an $X'$ exists
if and only if the
following statement is true: \emph{If a homogeneous polynomial $H(X')$
is zero on the set
$X_2 + X_3 + \cdots + X_n = 1$, then $H$ is identically zero}.  To prove this
statement, note that an open set of points in $\R^{n-2}$ can be written as $tX'$
for some $X'$ on the hyperplane $X_2+\cdots+X_n=1$.  Using homogeneity
if $H$ as $H(tX') = t^k H(X')$ the claim is proved.

By the claim it is easy to see that $\tilde{P}$ is indecomposable if
$P$ was indecomposable, since no nodes can ``disappear'' by a proper choice
of $X'$.  Obviously $N(P) \geq N(\tilde{P})$ and hence it suffices to
estimate the p-degree of $\tilde{P}$.

It is possible that even after every possible reordering of the variables,
the p-degree of $\tilde{P}$ is strictly smaller than the p-degree of $P$.
What we can do however is to reorder the variables to make the p-degree drop
the least.  That is, it is easy to see that we can pick an ordering of
variables such that $a_1 \geq \frac{d}{n-1}$ where $d =
\operatorname{p-degree}(P)$.  Since monomials of $P$ had no common divisor,
then the only possible common divisor of $\tilde{P}$ is $T^k$ for some $k$.
As $a_1 \geq \frac{d}{n-1}$ then
$k \leq d - \frac{d}{n-1}$, or in other words
the p-degree of $\tilde{P}$ must be at least $\frac{d}{n-1}$.  Therefore we
estimate (using the 2-dimensional result).
\begin{equation}
\operatorname{p-degree}(P)
\leq (n-1)\operatorname{p-degree}(\tilde{P})
\leq (n-1)(2N(P)-5) .
\end{equation}
\end{proof}

We can also use the ideas in this paper to prove the following interesting
result, which suggests the sharp bounds in $\sH(n,d)$ for all $n$.

\begin{thm}\label{elim-thm}
Let $p \in \sH(n,d)$, $n \geq 3$, then for all $j = 1,\ldots,n$,
at least $d$ terms of $p$ depend on $x_j$.
\end{thm}

Note that such a theorem is not true in $n=2$ and the proof suggests why.

\begin{proof}
Suppose that $j=1$.  We will construct two 2-dimensional Newton diagrams
where one variable corresponds to $x_1$ and the other variable corresponds to
all the other variables.  Let us assume that $n=3$.  The proof for $n \geq 3$
then follows by setting variables equal to each other.

We could ``view'' the 3-dimensional diagram from different ``angles'' to
obtain 2-dimensional diagrams.  More precisely take a 3-dimensional
diagram $D$ induced by $p$.  We define a 2-dimensional diagram $D_1$
as follows.  We let $(a,b)$ in $D_1$
be zero if $(a,k,b-k)$ is zero in $D$ for all $k$.  Otherwise, we find
the smallest $k$ such that
$(a,k,b-k)$ is nonzero.  We let $(a,b)$ in $D_1$ equal the value of
$(a,k,b-k)$.
We define $D_2$ similarly, but we take the largest possible $k$ above.
See Figure~\ref{fig:newtonview}.

\begin{figure}[h!t]
\begin{center}
\includegraphics{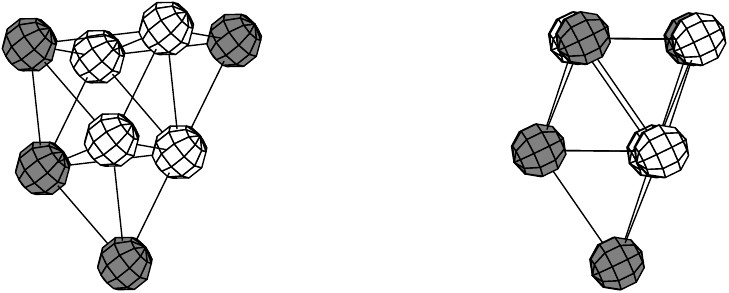}
\caption{Defining a 2-dimensional Newton diagram as a ``view'' of
a 3-dimensional diagram.\label{fig:newtonview}}
\end{center}
\end{figure}

As $p \in \sH(3,d)$, we note that the diagrams $D_1$ and $D_2$ are
diagrams corresponding to polynomials in $\sH(2,d)$.  That is,
there would be no negative terms.

We look at the 2-dimensional diagram $D_j$, $j=1,2$.
We recall a fact noticed by D'Angelo, Kos and Riehl~\cite{DKR},
that we used before.
We note that
there can only be two nodes on the edge $(0,b)$, otherwise
the diagram could not correspond to a polynomial in $\sH(2,d)$ (a negative
term would appear).

By Lemma~\ref{sharper-lemma}
we have the estimate $SC(D_j) \geq \frac{d+1}{2}$.  Not counting
the vertex node on the edge $(0,b)$, this number drops to $\frac{d}{2}$.
It is not hard to see that
each face node in $D_1$ corresponds to a node in $D$ that
does not correspond to any node in $D_2$.  On the other hand, an edge or
a vertex node in $D_1$ could correspond to the same node in $D$
as does an edge or a vertex node in $D_2$.

Hence, there must be at least $2 \frac{d}{2} = d$ nodes in $D$
corresponding to monomials $x_1^ax_2^bx_3^c$ with $a \geq 1$,
that is, nodes that correspond to
terms that depend on the variable $x_1$.  This completes the proof.
\end{proof}

In the case when $p$ has a monomial that depends on at most three variables
we immediately obtain the sharp estimate.

\begin{cor}\label{elim-corollary}
Let $p \in \sH(n,d)$, $n \geq 3$ and assume that $p$ has a monomial term
of degree $d$ that depends on at most three variables. Then
\begin{equation}
d \leq \frac{N(p) - 1}{n-1}.
\end{equation}
\end{cor}

\begin{proof}
The statement follows by induction on $n$. When $n=3$ we have already proved
the estimate.  Without loss of generality assume
that $p$ contains a monomial $x_1^a x_2^b x_3^c$ with $a+b+c=d$.
When $n \geq 4$ we can set $x_k = 0$ for some $k > 3$. By
Theorem \ref{elim-thm} we lose at least $d$ terms and by the assumption
that $p$ has the term $x_1^a x_2^b x_3^c$ the degree remains $d$.
Applying the induction
hypothesis for $n-1$ variables completes the proof.
\end{proof}

Unfortunately for a general polynomial the degree may drop when a variable is eliminated which means the proof of Corollary \ref{elim-corollary} does not work for all $p \in \sH(n,d)$.

%%%%%%%%%%%%%%%%%%%%%%%%%%%%%%%%%%%%%%%%%%%%%%%%%%%%%%%%%%%%%%%%%%%%%%%%%%%%%

\section{Connection to CR geometry} \label{section:CR}

As we mentioned in the introduction,
the primary motivation for this work
comes from CR geometry.  Let us
describe this connection in detail.

A basic question in CR geometry is to classify CR maps $f \colon M \to M'$
of CR manifolds $M$ and $M'$ of possibly different dimensions.  The sphere
and the hyperquadrics are the simplest examples of CR manifolds.  Under
enough regularity assumptions, for general $M$ and $M'$ few if any such maps
exist.
However,
the sphere and the hyperquadrics admit a lot of symmetry that
allows many inequivalent maps to exist.
Much effort has
been expended in the CR geometry community
in understanding the case when one or both of $M$ and $M'$
are spheres and hyperquadrics.
For
the sphere case, see for
example~\cites{L:hq23,DLP,DKR,HJX,DAngelo:CR,D:prop1,D:poly,DL:families,DL:complex}
and the references within.  For the general hyperquadric case
see~\cites{BH,HJX,BEH}
and the references within.

First define the hyperquadric $Q(a,b) \subset \C^n$, where $a+b=n$, by
\begin{equation}
Q(a,b) := \{ z \in \C^n \mid \sum_{k=1}^a \abs{z_k}^2 -
\sum_{k=a+1}^{a+b} \abs{z_k}^2
= 1 \} .
\end{equation}
Note that $Q(n,0)$ is the sphere $S^{2n-1}$.  We can also regard $Q(a,b)$
as a subset of the complex projective space $\CP^n$ by replacing
the 1 on the right hand side by $\abs{z_0}^2$.

Given any real algebraic manifold $M \subset \C^n$ we can take a
real polynomial $p$ vanishing on $M$ and write
\begin{equation}
p(z,\bar{z}) = \norm{f(z)}^2 - \norm{g(z)}^2 ,
\end{equation}
where $f$ and $g$ are mappings to some Hilbert spaces $\C^a$ and $\C^b$
and $\norm{\cdot}$ denotes the standard Hilbert space norm on $\C^a$ and $\C^b$.
If we homogenize the map $(f,g)$ we get a homogeneous polynomial map
to $\C^{a+b}$.
Hence $p$ gives rise to a rational CR map of $M$ to $Q(a,b-1) \subset
\CP^{a+b-1}$.
See
\cites{DAngelo:CR,L:hq23} for more on this construction.

When $M = M' = S^{2n-1} \subset \C^n$, $n \geq 2$,
then any nonconstant CR map $f$ must be a linear
fractional automorphism of the unit ball~\cite{Pincuk}.
Faran~\cite{Faran:B2B3} classified all sufficiently smooth
maps of $S^3$ to $S^5$.
The first author~\cite{L:hq23} recently
finished Faran's classification for all rational hyperquadric maps
in dimensions $2$ and $3$.
Forstneri\v{c}~\cite{Forstneric} proved that any $C^\infty$ CR map
of spheres is rational and that
the degree of $F$ is bounded in terms of $n$ and $N$.

When every component of a map is a single monomial, we say the map is
monomial.  While there exist maps that are not equivalent to
monomial maps, the monomial maps are a rich subset of examples, which exhibit
much of the combinatorial difficulty of the classification problem.
Given a CR map $f \colon S^{2n-1} \to S^{2N-1}$ for
$n \geq 3$ and the codimension $N-n$ sufficiently small, all
such maps are equivalent to quadratic monomial maps.
See \cite{Faran:lin} and others,
most recently \cite{HJX}.  The first author proved
in \cite{L:hq23} that all quadratic $f$ are equivalent to monomial maps.
While there exist degree 3 rational maps that are not equivalent
to monomial maps~\cites{FHJZ,L:hq23}, all the known examples are at least
homotopic to a monomial map.

D'Angelo
classified the polynomial CR sphere maps~\cite{D:poly} in the
following sense.  If we allow the target dimension to be large enough,
all polynomial
proper maps are obtained by a finite number of operations from
the unique homogeneous sphere map of same degree.
See also the survey \cite{Forstneric:survey}
or the book \cite{DAngelo:CR} for
more information about the problem.
D'Angelo conjectured that if $f \colon S^{2n-1} \to S^{2N-1}$ is a rational
CR map, then
\begin{equation} \label{eqcrbnd}
\deg f \leq
\begin{cases}
2N-3 & \text{ if $n=2$} \\
\frac{N-1}{n-1} & \text{ if $n \geq 3$}
\end{cases} .
\end{equation}
If true, this bound is sharp; there exist maps that achieve equality in the
bound.  In fact, there exist monomial maps that achieve equality.
No such bound exists when $n=1$; $z \mapsto z^d$ takes
$S^1$ to $S^1$ and is of arbitrary degree.

In \cite{DKR}, D'Angelo, Kos, and Riehl proved the bound \eqref{eqcrbnd}
for monomial maps when $n=2$.  One of the main results of this paper
extends their results to $n=3$.

To extend such bounds to hyperquadric maps, we need an extra condition.  It
is not hard to construct unbounded hyperquadric maps, for example
\begin{equation} \label{reduciblemapex}
[z_0,z_1,z_2] \mapsto
[z_2^{d-1}z_0, z_2^{d-1}z_1,
z_0^d, z_0^{d-1}z_1, z_2^d, z_0^{d-1}z_2]
\end{equation}
is the homogenized rational CR map taking $S^3$ to
$Q(4,1)$ and is of arbitrary degree $d$.  We simply compose the map
with the defining equation for $Q(4,1)$ in homogeneous coordinates
for $\CP^5$ and find
\begin{multline}
\abs{z_2^{d-1}z_0}^2+\abs{z_2^{d-1}z_1}^2+\abs{z_0^d}^2
+\abs{z_0^{d-1}z_1}^2-\abs{z_2^d}^2-\abs{z_0^{d-1}z_2}^2
=
\\
=
(\abs{z_2}^2)^{d-1}(\abs{z_0}^2+\abs{z_1}^2-\abs{z_2}^2)
+
(\abs{z_0}^2)^{d-1}(\abs{z_0}^2+\abs{z_1}^2-\abs{z_2}^2) .
\end{multline}
The map
$z \mapsto [z_2^{d-1}z_0, z_2^{d-1}z_1, z_2^{d}]$ is equivalent
to the map
$z \mapsto [z_0, z_1, z_2]$.  Similarly
$z \mapsto [z_0^{d}, z_0^{d-1}z_1, z_0^{d-1}z_2]$ is also equivalent
to $z \mapsto [z_0, z_1, z_2]$.
Therefore, in homogeneous coordinates, the map \eqref{reduciblemapex}
is simply the direct sum of two representatives of the linear map.
And the linear map of course takes $S^3$ to $S^3$.
Of course the map depends on which representatives we pick.  For different
representatives we get different maps.
We will call maps such as the map \eqref{reduciblemapex} decomposable.

\begin{defn} \label{def:indecomposable}
A rational map $f \colon Q(a,b) \to Q(c,d)$ is said to be \emph{decomposable}, if
the induced homogeneous polynomial map $F$ can be decomposed as
\begin{equation}
F = G \oplus H ,
\end{equation}
where $G$ and $H$ are homogeneous polynomial maps that induce
rational maps
$g \colon Q(a,b) \to Q(c_1,d_1)$ and
$h \colon Q(a,b) \to Q(c_2,d_2)$.
If $f$ is not decomposable, it is said to be \emph{indecomposable}.

If $f$ is a monomial map (that is, if we can pick homogeneous coordinates
such that every component of $F$ is a monomial), and there exist monomial
$G$ and $H$ as above, then
$f$ is said to be \emph{monomial-decomposable}.
If $f$ is not monomial-decomposable, it is said to be
\emph{monomial-indecomposable}.
\end{defn}

It is not hard to see that a CR map of spheres must be indecomposable.
Therefore, results that apply to indecomposable hyperquadric maps in general
carry over to maps of spheres.

We now connect monomial CR maps to the real algebraic setup
of this paper.  Suppose that a monomial CR map $f \colon Q(a,b) \to Q(c,d)$
is written in
homogeneous coordinates.
\begin{equation}
z \mapsto [f_1(z),\ldots,f_{N+1}(z)].
\end{equation}
As $f$ is monomial, assume that for some multi-index $\alpha_k$
we have
\begin{equation}
f_k(z) = C_{\alpha_k} z^{\alpha_k} .
\end{equation}
Take a defining function in homogeneous coordinates for the
target hyperquadric $Q(c,d)$
\begin{equation}
\sum_{k=1}^c \abs{w_k}^2 -
\sum_{k=c+1}^{c+d+1} \abs{w_k}^2 .
\end{equation}
Now compose $f$ with this defining function to obtain
\begin{equation}
\sum_{k=1}^c \abs{f_k(z)}^2 -
\sum_{k=c+1}^{c+d+1} \abs{f_k(z)}^2 .
\end{equation}
By replacing $\abs{z_0}^2$ with $x_0$,
$\abs{z_1}^2$ with $x_1$, and so on, we obtain
\begin{equation} \label{realpolyversioneq}
\sum_{k=1}^c \abs{C_{\alpha_k}}^2 x^{\alpha_k}
 -  \sum_{k=c+1}^{c+d+1} \abs{C_{\alpha_k}}^2 x^{\alpha_k} .
\end{equation}
The defining equation for the source hyperquadric $Q(a,b)$ becomes
\begin{equation} \label{sourcedefeq}
\sum_{k=1}^a x_k -
\sum_{k=a+1}^{a+b+1} x_k .
\end{equation}
Hence, \eqref{realpolyversioneq} is zero whenever
\eqref{sourcedefeq} is zero.

If none of the monomials of $f$ are repeated (that is if
two components of $f$ are the same monomial up to a constant multiple),
it is not hard to see that the degree of
\eqref{realpolyversioneq} is equal to the degree of $f$.

\begin{prop} \label{CRmontoreal:prop}
Let $f \colon Q(a,b) \to Q(c,d)$ be a monomial CR map.  Suppose that
the components of $f$ are linearly independent.
\begin{enumerate}[(i)]
\item \label{CRmontoreal:prop:i}
The degree of \eqref{realpolyversioneq} is equal to the degree of $f$.
\item \label{CRmontoreal:prop:ii} The map $f$ is monomial-indecomposable
if and only if
\eqref{realpolyversioneq} is indecomposable in the sense of
Definition~\ref{defn:basicdefs}.
\end{enumerate}
\end{prop}

\begin{proof}
To prove \eqref{CRmontoreal:prop:i}, notice that for a monomial
map the components of $f$ are linearly independent if and only if
no monomial of $f$ is repeated.  Then it is obvious that the degree
of $f$ must be equal to the degree of the induced polynomial
\eqref{realpolyversioneq}.

To prove \eqref{CRmontoreal:prop:ii}, notice that as the components of $f$
are linearly independent, each component of $f$ produces a different monomial
in \eqref{realpolyversioneq}.  The statement follows.
\end{proof}

Using Proposition~\ref{CRmontoreal:prop} we can apply results of this paper
to monomial CR maps of hyperquadrics.  Therefore the proof of
Theorem~\ref{thm:CRbounds} is a trivial application of this Proposition
and the results of this paper.  In particular, part \eqref{thm:CRbounds:i}
of Theorem~\ref{thm:CRbounds} follows by Theorem~\ref{maindim2thm},
part \eqref{thm:CRbounds:ii} follows by Theorem~\ref{thm:nonhomog},
and finally part
\eqref{thm:CRbounds:iii} follows by Theorem~\ref{thm:eekbound2}.

Therefore, at least for monomial maps, indecomposability appears to be
the correct condition to require to obtain degree bounds for
general hyperquadric maps.

%%%%%%%%%%%%%%%%%%%%%%%%%%%%%%%%%%%%%%%%%%%%%%%%%%%%%%%%%%%%%%%%%%%%%%%%%%%%%

%FIXME: else I don't get links, weird
%\renewcommand\MR[1]{\relax\ifhmode\unskip\spacefactor3000 \space\fi
  %\def\@tempa##1:##2:##3\@nil{%
    %\ifx @##2\@empty##1\else\textbf{##1:}##2\fi}%
  %\href{http://www.ams.org/mathscinet-getitem?mr=#1}{MR \@tempa#1:@:\@nil}}
\def\MR#1{\relax\ifhmode\unskip\spacefactor3000 \space\fi%
  \href{http://www.ams.org/mathscinet-getitem?mr=#1}{MR#1}}

\end{document}